\documentclass[a4paper, 11pt]{article}
\usepackage{mathrsfs,amssymb,amsmath}

\usepackage[top=1.31in,bottom=1.29in,left=1.1in,right=0.9in]{geometry}

\usepackage{amsthm}
\newtheoremstyle{thm}{0.55em}{0.8ex}{}{}{\bf}{.}{0.55em}{\thmname{#1}\thmnumber{ #2}\thmnote{(#3)}}

\theoremstyle{thm}
\numberwithin{equation}{subsection}
\newtheorem{Theorem*} {Theorem}
\newtheorem{Defn*} {Definition}
\newtheorem{Prop*}{Proposition}
\newtheorem{Lemma*}{Lemma}
\newtheorem*{conj}{Conjecture}

\newtheorem*{theo}{Theorem}

\swapnumbers
\newtheorem{Defn}[equation]{Definition}
\newtheorem{Example}[equation]{Example}
\newtheorem{Remark}[equation]{Remark}
\newtheorem{Remarks}[equation]{Remarks}
\newtheorem{Prop}[equation]{Proposition}
\newtheorem{Theorem}[equation]{Theorem}
\newtheorem{Notation}[equation]{Notation}

\newtheorem{Lemma}[equation]{Lemma}
\newtheorem{Cor}[equation]{Corollary}

\def\Case#1{\medskip\noindent\textbf{Case #1}.\leavevmode\newline}

\makeatletter
\def\enumerate{\begingroup\ifnum\@enumdepth>3\@toodeep\else
      \advance\@enumdepth\@ne
      \edef\@enumctr{enum\romannumeral\the\@enumdepth}%
      \topsep\z@\parskip\z@
      \list{\csname label\@enumctr\endcsname}
        {\@nmbrlisttrue\let\@listctr\@enumctr
         \parsep\z@\itemsep\z@\topsep\z@
         \setcounter{\@enumctr}{0}
         \def\makelabel##1{\hss\llap{\rm ##1}}
       }\fi}

\makeatother
\def\({\big(}
\def\){\big)}

\def\C{\mathbb C}
\def\N{\mathbb N}
\def\R{\mathbb R}
\def\Q{\mathbb Q}
\def\Z{\mathbb Z}
\def\F{\mathbb F}

\def\E{\mathcal E}

\def\c{\mathbf C}
\def\r{\mathbf R}
\def\h{h}
\def\up{l}
\def\ri{ {a}}
\def\J{\mathfrak{J}}

\def\O{\mathcal O}

\def\B{\mathcal B}

\def\t{\mathfrak t}
\def\u{\mathfrak u}
\def\s{\mathfrak s}
\def\p{\mathfrak p}
\def\pf{\p_{_{\mathbf f}}}
\def \po{\p^{\mathbf o}}
\def \pfo{\p^{\mathbf o}_{_{\mathbf f}}}
\def \ppi{\p_{_{\mathcal I}}}
\def \pj{\p_{\!_{\mathcal J}}}
\def\tt{\underline{\t}}

\def\ud{\u_{\mbox{\tiny $d$}}}
\def\uu{\underline{\ud}}

\def\X{\mathfrak X}
\def\J{\mathfrak J}
\def\M{\mathfrak M}

\def\SC{\mathfrak S}
\def\CC{\mathfrak C}
\def\R{\mathfrak R}

\DeclareMathOperator{\Hom}{Hom}
\DeclareMathOperator{\tab}{tab}
\DeclareMathOperator{\Res}{res}
\DeclareMathOperator{\RRes}{Res}
\DeclareMathOperator{\Ind}{Ind}

\DeclareMathOperator{\Span}{-\,span}

\DeclareMathOperator{\row}{row}

\DeclareMathOperator{\Std}{Std}
\DeclareMathOperator{\RStd}{RStd}

\DeclareMathOperator{\last}{last}
\DeclareMathOperator{\ttop}{top}
\DeclareMathOperator{\Stab}{Stab}
\DeclareMathOperator{\Char}{char}
\DeclareMathOperator{\rank}{rank}

\begin{document}
\setlength{\abovedisplayskip}{4pt}
\setlength{\belowdisplayskip}{4pt}

\title{ On the U-module Structure of the Unipotent \\Specht Modules of Finite General Linear Groups  }
\date{September 3, 2012}
\author{Qiong Guo\\ Institut f\"{u}r Algebra und Zahlentheorie\\ Universit\"{a}t Stuttgart, 70569 Stuttgart, Germany
\\ \scriptsize{E-mail: guo.qiong@mathematik.uni-stuttgart.de}}


\sloppy \maketitle

\begin{abstract}
Let $q$ be a prime power, $G=GL_n(q)$ and let $U\leqslant G$ be the subgroup of (lower) unitriangular matrices in $G$. For a partition $\lambda$ of $n$ denote the corresponding unipotent Specht module over the complex field $\C$ for $G$ by $S^\lambda$. It is conjectured that for $c\in \Z_{\geqslant 0}$  the number of irreducible constituents
of dimension $q^c$ of  the restriction $\RRes^{G}_U(S^\lambda)$ of $S^\lambda$ to $U$ is a polynomial in $q$ with   integer coefficients depending only on $c$ and $\lambda$, not on $q$. In the special case of the partition $\lambda=(1^n)$ this implies a longstanding (still open) conjecture of Higman \cite{higman}, stating that the number of conjugacy classes of  $U$ should be a polynomial in $q$ with integer coefficients depending only on $n$ not on $q$.    In this paper we prove the  conjecture in the case that $\lambda=(n-m,m)$ $(0\leqslant m \leqslant n/2)$ is a 2-part partition. As a consequence, we obtain a new representation theoretic construction of the standard basis of $S^\lambda$ (over fields of characteristic coprime to $q$) defined by M. Brandt, R. Dipper, G. James and S. Lyle in \cite{brandt2}, \cite{dj1} and an explanation of the rank polynomials appearing there.
\end{abstract}

\section{Introduction}
let $p\in \N$ be a prime, $\F_q$ the finite field with $q$ elements, where $q$ is a power of $p$. Let $F$ be a field whose
characteristic is coprime to $p$ and which contains a primitive $p$-th root of unity. Let $U=U_n(q)$ be
the group of lower unitriangular $n\times n$-matrices with
entries in $\F_q.$ Thus $U$ is a $p$-Sylow subgroup of the
general linear group $G=GL_n(q)$.

It follows from \cite{huppert} and \cite{Isaacsq} that
every irreducible complex character of $U$ has degree a power of $q$.
There is a long standing conjecture, contributed to Higman
(c.f. \cite{higman}) stating, that there should be polynomials
$h_n(t)\in \Z[t]$ such that $h_n(q)$ is the number of
conjugacy classes of  $U.$
By general theory $h_n(q)$ equals the number of distinct
irreducible complex characters of $U,$ and hence Higman's
conjecture immediately follows from the following conjecture: \vspace{-1mm}
\begin{conj}[G. Lehrer 1974, \cite{lehrer}]\label{lehre}
For $0\leqslant c \in \Z,\,n\in \N$ there exists
$l_{n,c}(t)\in \Z[t]$ such that $l_{n,c}(q)$ is the number
of distinct irreducible complex characters of degree $q^c$ of
$U.$
\end{conj}


Isaacs put forward another stronger conjecture in
\cite{Isaacs2}:
\vspace{-1mm}
\begin{conj}[Isaacs]
$l_{n,c}(t)$ is a polynomial in $(t-1)$
with non-negative integer coefficients.
\end{conj}


There is a remarkable set of $FG$-modules  called unipotent Specht modules, defined for all fields $F$ of characteristic coprime to $q$. These are labeled by
partitions $\lambda$ of $n$, $\lambda \leftrightarrow S^\lambda_F$, and for $F=\C$ these are precisely the distinct irreducible constituents of the permutation representation of $G$ on the cosets of a Borel subgroup $B\leqslant G$ (for instance $B$ is the set of invertible upper triangular $n\times n$-matrices).

In this paper we shall show the unipotent Specht module $S_F^\lambda$ for $G$  for a 2-part partition $\lambda$ of $n$ restricted to the lower unitriangular group $U$ satisfies a kind of Isaacs' conjecture, which we believe to hold for all
unipotent
Specht modules, i.e. all partitions $\lambda$ of $n$: \vspace{-1mm}
\begin{conj}For each $c\in \Z_{\geqslant 0}$
there exists a polynomial $d_{c,\lambda}(t)\in \Z[t]$  such
that
$d_{c,\lambda}(q)$  is the number of irreducible constituents
of dimension $q^c$ of $\RRes^{G}_U(S_F^\lambda)$. Moreover, $d_{c,\lambda}(t)$ is a polynomial in $(t-1)$ with  non-negative integer coefficients.
\end{conj}

In particular, in the case $\lambda=(1^n)$, the
corresponding unipotent Specht module $S^\lambda$ is the
Steinberg module. It is known that in this case
$\RRes^{G}_U S^\lambda$
is the regular $U$-module. Hence the conjecture above specialized to the case $\lambda=(1^n)$ implies Issacs' conjecture and hence Higman's conjecture.
It is known that classifying the conjugacy classes of $U$ is a wild problem and hence classifying the irreducible complex characters of $U$ seems to be a wild problem as well. However
 C. A. M. Andre   and subsequently N. Yan discovered a remarkable new decomposition of the regular character of $U$ into a set of orthogonal characters, called supercharacters in \cite{andre1}, \cite{yan}. This notion was subsequently axiomatized by P. Diaconis and I. M. Issacs and applied to $\F_q$-algebras.  Yan constructed a monomial basis called Fourier basis,  for $\C[U]$, the space of
complex-valued functions on $U$.
In this paper, we consider first the restriction to $U$ of the permutation representation of $G$ on the cosets of the standard parabolic subgroup $P_\lambda$ in $G$ where $\lambda$ is a composition of $n$. By Mackey's decomposition theorem this splits into submodules labeled by row standard $\lambda$-tableaux, called batches. Each batch has a Fourier type basis, called idempotent basis, on which a certain subgroup of $U$ acts monomially. We shall not carry this out in full generality, but restrict ourselves to the special case of two part partitions $\lambda$. However we point out that for the special case $\lambda=(1^n)$ and the unique batch attached to the only standard $\lambda$-tableau, our idempotent basis is dual to Yan's Fourier basis.

Exploring basic properties of idempotent bases we obtain as a consequence a new, representation theoretic proof of the following standard basis conjecture for unipotent Specht modules in the special case of $\lambda=(n-m,m)\vdash n, \, 0\leqslant m \leqslant n/2$:

\begin{conj}[Dipper-James, 1990] Let $\lambda\vdash n$. Then  there exists for each $\s\in \Std(\lambda),$  a polynomial $r_\s(t)\in \Z[t]$  and a subset
$\mathcal B_\s\subset S^\lambda $ independent of $q$ and $F$ such that the following holds: \vspace{-2mm}
\begin{itemize}
 \item[(1)] $r_\s(1)=1$
\vspace{-2mm}
 \item[(2)] $|\mathcal B_\s|=r_\s(q)$
\vspace{-2mm}
 \item[(3)] The union $\mathcal B=\mathcal B^\lambda=\bigcup_{\s\in \Std(\lambda)}\mathcal B_\s$ is disjoint.
\vspace{-2.5mm}
 \item[(4)] $\mathcal B$ is a basis of $S^\lambda.$
\vspace{-1mm}
\end{itemize}
The polynomials $r_\s(t)$ are called rank polynomials and the basis $\mathcal B$  of $S^\lambda$ is called the {\bf standard basis} of $S^\lambda.$
\end{conj}

This conjecture was proved by M. Brandt, R. Dipper, G. James and S. Lyle for the case
that $\lambda$ is a 2-part partition  in \cite{brandt}, \cite{dj1}. But the proof there is rather combinatorial hence our new representation theoretic proof seems to open up a new way to solve this conjecture for arbitrary partition $\lambda$ of $n$. In particular, we give a representation explanation of those rank polynomials.

We now fix some notation which is used throughout this paper.
 We identify the set
$\Phi=\{(i,j)\,|\,1\leqslant i, j \leqslant n, i\not=j\}$ with
the standard root system of $G$ where
$\Phi^+=\{(i,j)\in \Phi \,|\,i>j\}$, $\,\Phi^-=\{(i,j)\in \Phi
\,|\,i<j\}$ are the positive respectively negative roots with
respect to the basis $\Delta=\{(i+1,i)\in
\Phi^+\,|\,1\leqslant i \leqslant n-1\}$ of $\Phi$.
A subset $J$ of $\Phi$ is {\bf closed} if $(i,j),(j,k)\in J, (i,k)\in \Phi$ implies
$(i,k)\in J.$
For $1\leqslant i, j \leqslant n$ let $\epsilon_{ij}$ be the
$n\times n$-matrix $g=(g_{ij})$ over $\F_q$, with $g_{ij}=1$ and
$g_{kl}=0$ for all $1\leqslant k,l \leqslant n$ with
$(k,l)\not=(i,j).$ Thus $\{\epsilon_{ij}\,|\, 1\leqslant i, j
\leqslant n \}$ is the natural basis
of the $\F_q$-algebra $M_n(\F_q)$ of $n\times n$-matrices with
entries in $\F_q.$
For $1\leqslant i, j \leqslant n,$\, $i\not=j$ and $\alpha\in
\F_q$, let $x_{ij}(\alpha)=E_n+\alpha \epsilon_{ij},$ where $E_n$
is   the $n\times n$-identity matrix. Then
$X_{ij}=\{x_{ij}(\alpha)\,|\,\alpha\in \F_q\}$ is the {\bf root
subgroup} of $G$ associated with
the root $(i,j)\in \Phi,$ and is isomorphic to the additive
group $(\F_q,+)$ of the underlying field $\F_q$, hence is in
particular abelian. Moreover $U=\langle
x_{ij}(\alpha)\,|\,1\leqslant j<i\leqslant n, \, \alpha\in
\F_q\rangle$ is the unitriangular subgroup of $G=GL_n(q)$
consisting of all lower triangular matrices with ones on the
diagonal. It is well known that for a closed subset
$J$ of $\Phi^+$, the set
$
U_J=\{u\in U\,|\,u_{ij}=0,\,\forall\, (i,j)\notin J\}$ is the
subgroup of $U$ generated by $X_{kl},\, (k,l)\in J$ and if we choose any linear ordering on $J$ then
$U_J=\{\prod_{(i,j)\in
J}x_{ij}(\alpha_{ij})\,|\,\alpha_{ij}\in \F_q\}$, where the
products are given in the fixed linear ordering. Note that,  $J \subseteq \Phi^+$ is closed if and only if $(i,j),(j,k)\in J$ implies $(i,k)\in J$.

Let $\lambda=(\lambda_1, \lambda_2, \cdots, \lambda_h)$
be a composition of $ n.$ Then a set of subspaces $V_0, V_1,V_2,\cdots,
V_h$ of the vector space $ \F_q^n$  with the properties
$V=V_0 \supseteq V_1 \supseteq \cdots \supseteq V_{h-1} \supseteq V_h=0$ such that $
\dim(V_{i-1}/V_i)=\lambda_i,\, \forall\,1\leqslant i \leqslant h$ is called a {\bf $\lambda$-flag}.
The set of $\lambda$-flags is denoted by $\mathcal F(\lambda).$
Clearly, right multiplication of $G$ on $V$ induces a permutation action of $G$ on $\mathcal F(\lambda).$
The corresponding {\bf permutation module} is denoted by $M^\lambda$. It is easy to see that
$M^\lambda=\Ind^{{\tiny{\mbox {$G$}}}}_{{\tiny{\mbox {$P_\lambda$}}}} F$, where $P_\lambda$ is the {\bf standard parabolic subgroup} of $G$ with respect to $\lambda$, containing $U^-$, the group  of upper unitriangular matrices in $G$ and $F=F_{P_\lambda}$ is the trivial $FP_\lambda$-module.
If $\Char (F)=0$, the unipotent Specht modules $S^\lambda $ vary over  pairwise non-isomorphic irreducible modules for $G$.   Moreover, Gordon James gave for fields $F$
with $\Char(F)\not=p$, the following characteristic free description of unipotent Specht modules analogous to the theory of Specht modules for symmetric groups:

\begin{theo}
If $\lambda$ is a composition of $n,$ then the unipotent Specht module associated with $\lambda$ is given as
$$S_F ^\lambda=\bigcap\limits_{\mu\rhd \lambda}\{\ker \Phi: \Phi\in \Hom_{FG}(M^\lambda,M^\mu)\}.$$
Here $\rhd$ is the usual dominance order. Moreover, $S^\lambda _ \C$ is irreducible and for $\Char(F)=l\not=p$, $S^\lambda_F$ is a reduction modulo $l$ of $S^\lambda_\C.$
\end{theo}

\section{$U$-module structure of $M^{(n-m,m)}$}
The kernel intersection theorem  suggests that it may be a good idea, to inspect first the restriction
of the permutation module $M^\lambda$ to $U$, of which $\RRes^{\tiny{\mbox {$G$}}}_{\tiny{\mbox {$U$}}} S^\lambda $
is a submodule.

\subsection{Normal form of a $(m\times n)$-matrix}\label{permutation}
Let $\lambda= (n- m, m)\vdash n$ (thus
$0\leqslant m \leqslant n/2$). Then
$\mathcal F (\lambda)=\{0\subseteq V_1\subseteq V=\F_q^n\,|\,\dim_{\F_q} V_1=m\}$. We list a basis of $V_1$ as \mbox{$m\times n$}-matrix  and then row reduce it to a unique normal form defined as follows (comp. \cite{brandt}, \cite{dj1}):
\begin{Defn}\label{Xi}
Let $m, n$ be integers with $0 \leqslant m \leqslant n.$ Denote by $\Xi_{m,n}$ the set of $m\times n$
matrices $L=(l_{b_ij}) $ over $\F_q$ with the property that for some integers $b_1,\cdots,b_m$ with
$1\leqslant  b_1 < b_2 < \cdots< b_m \leqslant n$ the following holds for each
$i, \text{ with } 1\leqslant i \leqslant m:$
\begin{itemize}
\item[(1)] $l_{b_ib_i}=1,$ and $l_{b_ij}=0$ if $j>b_i;$
\item[(2)] $l_{b_kb_i}=0\text{ if }k>i.$
\end{itemize}
\end{Defn}

\begin{Remark} Note in the  definition above, we label the rows of the element in $\Xi_{m,n}$ by $b_1, b_2, \cdots, b_m$
instead of $1,2,\cdots,m.$ The reason for doing this will become apparent later on. Moreover, for each $i,$ $l_{b_ib_i}=1$
is the last nonzero entry in row $b_i.$ We call it {\bf ``last 1'' } for convenience.
\end{Remark}

Every $(m \times n)$-matrix over $ \F_q$ of rank $m$ is row-equivalent to precisely one matrix in $\Xi_{m,n}$.
Therefore $\Xi_{m,n}$ is in bijection with the set of m-dimensional subspaces of an $n$-dimensional vector space
over $\F_q$. Actually, the set $\Xi_{m,n}$ can be generalized to $\Xi_\lambda$ for arbitrary composition $\lambda$ of $n$ (see \cite{brandt}).


\begin{Defn}
\begin{itemize}
\item[(1)] If $m$ is a non-negative integer, then we let
$[m]=1+q+q^2+\cdots+q^{m-1}.$\vspace{-3mm}
\item[(2)]If $m, n$ are  non-negative integers, let
$$\begin{bmatrix}n\\m\end{bmatrix}=\begin{cases}
\frac{[n][n-1]\cdots[n-m+1]}{[m][m-1]\cdots[1]}& \text{if } n\geqslant m\\
0&\text{otherwise.} \end{cases}$$
\end{itemize}
Then $[\begin{smallmatrix}n\\m\end{smallmatrix}]$ is a polynomial in $q$, known as a Gaussian polynomial. Since $q$ is a prime power, $[\begin{smallmatrix}n\\m\end{smallmatrix}]$ is the number of $m$-dimensional subspaces of an $n$-dimensional vector space over $\F_q.$
\end{Defn}

\begin{Defn}\label{basisxi} \label{circ} Let $M^{(n-m,m)}$ be the $[\begin{smallmatrix} n\\m\end{smallmatrix}]$-dimensional\vspace{1mm}
vector space over F with basis $\Xi_{m,n}$. If $L\in \Xi_{m,n}$ and $g\in G $
then $Lg$ is row-equivalent to a matrix in $\Xi_{m,n}$, and we denote this matrix by $L \circ g$.
Under the action $\circ\text{ of } G$, the vector space $M^\lambda$ becomes an $FG $-module, $\lambda=(n-m,m)\vdash n.$
\end{Defn}

Obviously, this is isomorphic to the permutation module of $G $ on the cosets of the parabolic
subgroup for $\lambda$ defined previously justifying the notation.

Remember $U $ is the lower unitriangular subgroup of $G .$
Hence $M^\lambda$ can be regarded as an $FU $-module. Since $\Char(F)\not=p$ and $|U|$ is a $p$-power, $FU$ is semisimple.

\begin{Defn}\label{secondrow}Suppose that $L=(l_{b_ij} )\in \Xi_{m,n}$, and let $1\leqslant  b_1 < b_2 < \cdots< b_m \leqslant n$
be the integers which appear in Definition \ref{Xi}. Define  $\tab(L)$ to be unique the row-standard
$\lambda$-tableau whose second row is $b_1,b_2, \cdots ,b_m.$ We refer to $\tab(L)$ as the tableau of $L$.
\end{Defn}

We denote the set of row-standard $\lambda$-tableaux by $\RStd(\lambda)$. For $1\leqslant i \leqslant n,$ $\t\in \RStd(\lambda)$, let $\row_\s(i)$ be the row index of the row in $\t$ containing $i$. So for $\lambda=(n-m,m), \row_\t(i)\in\{1,2\}$ and we denote the second row of $\t$ by $\underline{\t}.$ Note that $\t$ is completely determined by $\underline{\t}.$ Naturally, we obtain
\mbox{$\underline{\tab(L)}=(b_1,b_2, \cdots ,b_m).$}

\begin{Example}Suppose
 $$\begin{matrix}
           &                             \begin{matrix}\,\,\emph 1&\,\,\emph 2&\emph 3&\emph 4\end{matrix}       & \\
 L=&\begin{pmatrix} l_{21}&1&0&0\\l_{31}&0&1&0\end{pmatrix}&\begin{matrix}\!\!\!\!\!\!\emph{2}\\\!\!\!\!\!\!\emph 3\end{matrix}
\, \in \Xi_{2,4} &\text{ then }\tab(L)=\begin{tabular}{|c|c|}\hline 1 & 4 \\\hline 2 & 3\\\hline \end{tabular}\,,\quad
\underline{\tab(L)}=(2,3).
\end{matrix}
$$

\end{Example}

\begin{Remark}
When $\lambda$ is a two part partition, we order the elements in
$\RStd(\lambda)$ lexicographically by their second rows.\end{Remark}

The positions in a matrix $M\in \Xi_{m,n}$, which are not in  columns of and  not to the right  of the last 1's
will play an important role in the following sections. And for the matrices having the same tableau, these positions are also
the same. Therefore we fix the following notation:
\begin{Defn}\label{jt}Set
$\J_{\t}=\{( i,j)\,|\, i>j,\, i\in\tt,\, j\notin \underline{\t}\,\}$ for $\t\in \RStd(\lambda)$ and $\underline{\t}=(b_1,b_2, \cdots, b_m).$
\end{Defn}
\vspace{2mm}
Since $\Xi_{m,n}$ is a basis of $M^\lambda$,  the following definition makes sense:\vspace{-1mm}
\begin{Defn}\label{original basis} Suppose that $v\in M^{\lambda},$ and write
$v =\sum_{X\in\Xi_{m,n}}C_XX$ where  $C_X\in F$ and \vspace{-1mm}
  \begin{itemize}
    \item [(1)]For each $\t\in \RStd(\lambda)$, let  $v(\t) =\sum_{\tab(X)=\t}C_XX .$\vspace{-2mm}
     \item [(2)] If $v\not=0,$ then let $\last(v)$ be the  last $\t\in \RStd(\lambda)$ (with respect to the lexicographical order as above)
           such that $v(\t)\not=0.$\vspace{-2mm}
    \item [(3)] For $v\not=0,$ define \mbox{$\ttop(v)=v\big(\last(v) \big) $.}
  \end{itemize}
\end{Defn}

\subsection{Idempotent basis}

Our first goal is to investigate the $U$-module structure of the permutation module $M^\lambda$. Obviously Mackey decomposition
provides a first splitting of $\RRes^{\tiny{\mbox {$FG$}}}_{\tiny{\mbox {$FU$}}} M^\lambda .$
Note that $\mathcal D_\lambda=\{w\in \SC_n\,|\,\t^\lambda w \in \RStd(\lambda)\}$ is a $P_\lambda$$-U $ double coset transversal in $G $. Note that this holds, even if $P_\lambda$ in our setting contains $U^-$, the group of upper unitriangular matrices. Thus
\[\RRes^{\tiny{\mbox {$FG$}}}_{\tiny{\mbox {$FU$}}}
M^\lambda=\RRes^{\tiny{\mbox {$FG$}}}_{\tiny{\mbox {$FU$}}}
\Ind^{\tiny{\mbox {$FG$}}}_{\tiny{\mbox {$FP_\lambda$}}} F=\bigoplus\nolimits_{w\in \mathcal D_\lambda}
\Ind^{\tiny{\mbox {$FU$}}}_{\tiny{\mbox {$F(P_\lambda^w\cap U)$}}} F\]
 is a direct sum decomposition of
$\RRes^{\tiny{\mbox {$FG$}}}_{\tiny{\mbox {$FU$}}} M^\lambda $. We call the $U$-submodule
$\Ind^{\tiny{\mbox {$FU$}}}_{\tiny{\mbox {$F(P_\lambda^w\cap U)$}}} F$ the {\bf $\t$-batch} of
$\RRes^{\tiny{\mbox {$FG$}}}_{\tiny{\mbox {$FU$}}} M^\lambda $, where $\t=\t^\lambda w\in \RStd(\lambda).$
We now translate this notion into the setting of section 2.1:
\begin{Lemma}\label{xj}Let  $\t=\t^\lambda w \in \RStd(\lambda).$ Set
\mbox{$\mathfrak{X}_{\t}= \{L\in \Xi_{m,n}\,|\,\tab(L)=\t \,\}.$}
Then for $L\in \X_\t$ and $u\in U,$ we have $L\circ u \in \X_\t.$ Moreover $U$ acts transitively on $\X_\t.$
Let $\M_{\t}$ be the corresponding permutation module with basis  $\X_{\t}$.
Then $\M_{\t}\cong\Ind^{\tiny{\mbox {$FU$}}}_{\tiny{\mbox {$F(P_\lambda^w\cap U)$}}} F\,, \,
\text{  the $\t$-batch of $\RRes^{\tiny{\mbox {$FG$}}}_{\tiny{\mbox {$FU$}}} M^\lambda $}.$
\end{Lemma}
\begin{proof}
For any $g\in U=\prod_{(i,j)\in \Phi^+} X_{ij}$, its circle action  on  $M\in \X_\t$ can be obtained firstly by a series of column operation from right to left, keeping the last 1's unchanged, and then using row operations to remove the possible nonzero entries under the last 1's. Therefore $\M_\t$ is an $U$-module under the operation $\circ$.

Next we show that $U$ acts transitively on $\X_\t.$
For this let $L=(l_{{b_ij}})\in\X_\t$, whose only nonzero entries are the last 1's. Then for any arbitrary $g\in U,\, Lg$ is obtained from $g$ by deleting all rows with index $j\notin \tt,$ then obviously we can easily construct $u\in U$ such that $L\circ u=M$ for any $M\in \X_\t.$ That is $U$ acts transitively on $\X_\t.$

To finish the proof it suffices  to show that the stabilizer
$\Stab_{\tiny{\mbox{$U$}}}(L)$ of $L$ in $U$ is given as $P_\lambda^ w\cap U.$
It is easy to see $L\circ u=L$ if and only if the entries in rows $b_i$ of $u$
are zeros except the   positions $(b_i,b_j)$ where $i\geqslant j$.
Then $\Stab_{\tiny{\mbox{$U$}}}(L)$ is generated by root subgroups $X_{ij}$ with
$1\leqslant j  <i \leqslant n$ where $i\notin \tt$ or $i,j\in \tt.$ Since $\t$ has precisely two rows, this condition is equivalent to
$1\leqslant j  <i \leqslant n\text{ and } \row_{\t}(i) \leqslant \row_{\t}(j)$ and we conclude
$\Stab_{\tiny{\mbox{$U$}}}(L)  = P_\lambda^ w\cap U$ (see \cite{dj1}).
\end{proof}

Next for $\t\in \RStd(\lambda)$ fixed, we make $\mathfrak{X}_{\t}$  into an abelian group through
introducing an addition $\diamond$ on $\mathfrak{X}_{\t}$ by
adding all entries pointwise besides the last one's.

\begin{Example}\label{diamond}Let $a_1,a_2,b_1,b_2,c_1,c_2 \in  \F_q.$ Then
$$\begin{pmatrix} a_1&1&0&0\\b_1&0&c_1&1\end{pmatrix}\diamond \begin{pmatrix} a_2&1&0&0\\b_2&0&c_2&1\end{pmatrix}=
\begin{pmatrix} a_1+a_2&1&0&0\\b_1+b_2&0&c_1+c_2&1\end{pmatrix}$$
\end{Example}

Obviously $(\mathfrak{X}_{\t},\diamond)$ is an abelian group of order $q^{|\mathfrak{J}_{\t}|}.$
Therefore we can find $q^{|\mathfrak{J}_{\t}|}$ linear irreducible $F$-characters of  $\X_{\t}.$ Such a character
$\chi$ is a group homomorphism from $\X_{\t}$ to the multiplicative group $F^{\ast}.$ In particular
\mbox{$\chi(M\diamond N)=\chi(M)\chi(N) \text{ for } M, N \in \X_{\t}.$}

We fix, once for all, a non trivial linear character
$\theta: \big (\F_q, + \big )\rightarrow F^\ast$. Following the notation in \cite{dj1}, we denote by $\xi_{_{b_ij}}$ the $(b_i,j)$
coordinate function from $\X_{\t}$ to $ \F_q$  for $(b_i,j)\in \J_{\t}$. For a given  matrix $L=(l_{b_ij})\in \X_{\t},$ we let
$\chi_{_L}=\sum_{(b_i,j)\in \J_{\t}}l_{_{b_ij}}\theta{\xi_{_{b_ij}}}$ so that
$X=\{\chi_{_L}\,|\,L\in \X_{\t}\}$ is the set of $F$-linear characters of $(\X_{\t},\diamond)$ as a vector space over $ \F_q$ and for $M=(m_{_{b_ij}})\in \X_{\t},$ we have
\begin{equation}\label{defofchi}\chi_{_L}(M)=\prod\nolimits_{(b_i,j)\in \J_{\t}}\theta(l_{_{b_ij}}m_{_{b_ij}})\end{equation}

Since  $\Char(F)\not=p$  and $|\X_\t|$ is a power of $p$, $F(\X_\t,\diamond)$ is  semisimple.
$F$ is a splitting field for $(\X_\t,\diamond)$ and $F(\X_\t,\diamond)$ has a basis of orthogonal primitive idempotents.
This basis turns out to be very well adapted to the $U$-module structure of $\M_\t$ as we shall show.

In order to not mix up the formal addition in the $F$-vector space $\M_\t$ and the   matrix addition $\diamond$, we write
$[M]$ if we consider the matrix $M$ as a basis element of the $F$-vector space $F(\X_{\t},\diamond)$.

\begin{Defn}\label{defofide}Suppose that $\t\in \RStd(\lambda)$ and $L \in \X_{\t}.$ Let
\begin{eqnarray*}
e_{_L}=\frac{1}{q^{|\mathfrak{J}_{\t}|}}\sum\limits_{M\in\X_{\t}}\chi_{_{L}}(-M)[M]
=\frac{1}{q^{|\mathfrak{J}_{\t}|}}\sum\limits_{M\in\X_{\t}}\prod\limits_{(b_i,j)\in\J_{\t}}\theta(-l_{_{b_ij}}m_{_{b_ij}})[M].
\end{eqnarray*}
\end{Defn}
By general theory $e_{_L}$ is the idempotent in $F(\X_{\t},\diamond)$ affording the linear character $\chi_{_L}.$ In fact,
\[\mathcal E_{\t}=\{e_{_L}\,|\,L \in \X_{\t}\}\] is a complete set of primitive orthogonal idempotents in $F(\X_{\t},\diamond),$ and so
\mbox{$F(\X_{\t},\diamond)= \bigoplus_{L \in \X_{\t}}Fe_{_L}$}is the decomposition of the regular module of $F(\X_{\t},\diamond)$
into pairwise non-isomorphic irreducible $F\X_{\t}$-modules. Since $\X_\t$ is an $F$-basis of $\M_\t$ too, we may consider the idempotents $e_{_L}$, $L\in \X_\t$ as elements of $\M_\t$,
and hence $\E_\t$ as $F$-basis of $\M_\t$, \ref{defofide} providing the base change matrix.

\subsection{The subgroup  $(U^w\cap U)$ of $U$}\label{intermod}
Next we introduce a subgroup of U, which will play an important role later on. That is,  $U^w\cap U=w^{-1}U w\cap U$. We remark in passing that $U^w\cap U$ is a set of left coset representatives of $P_\lambda^w \cap U$ in $U$.

\begin{Lemma}\label{defsub}Let $\lambda\vdash n, \,\s=\t^\lambda w\in \RStd(\lambda)$
 where $w=\text{d}(\s)\in \mathfrak S_n$. Let \mbox{$g=(g_{_{ij}})\in G$}. Then $g\in U^w\cap U$ if and only if
  $g\in U$ and $ \forall\, 1\leqslant i,j \leqslant n:$
  $(i<j \text{ or } \row_\s(i)<\row_\s(j)) $ implies $ g_{_{ij}}=0.$
So $U^w\cap U$ consists of all matrices, which are contained in $U$
and in addition have zeros at all places $(i,j)$ with $i>j$ and $\row_\s(i)<\row_\s(j).$
\end{Lemma}
\begin{proof}
Let $h=(h_{_{kl}})\in G$ and $g=(g_{_{ij}})=w^{-1}hw.$ Then
$$g_{_{ij}}=h_{_{kl}}\text{ for }i=kw, j=lw,\text{ and } \,\forall\,1\leqslant k,l \leqslant n.$$
The key of showing this argument is by using the following observation:
$i$ occupies the place in $\s$ which is occupied by $k$ in $\t^\lambda$ and
$j$ occupies the place in $\s $ which is occupied by $l$ in $\t^\lambda.$
\end{proof}

From now on we fix a 2-part partition \mbox{$\lambda=(n-m,m)\vdash n$} and $\t\in \RStd(\lambda).$
Let $w=d(\t)$ i.e. $\t^\lambda w=\t.$ Recall that the second row $\tt$ of $\t$ labels the rows of $L \in \X_{\t}$.
So let $\underline{\t}=(b_1,b_2,\cdots,b_m).$ In particular, we have:

\begin{Cor}\label{upsilon}  $g=(g_{_{ij}})\in U^w\cap U$ if and only if
$ g \in U$  and the following holds:
\mbox{$(i\notin \underline{\t} \text{ and } j\in \underline{\t}) \text{ implies }g_{_{ij}}=0.$} In particular, $U^w\cap U$
is generated by the root subgroups $X_{ij}$ where $1\leqslant j < i \leqslant n$ satisfying one of the following conditions:
 (1) $i \in \tt,\, j\notin\tt;$ \,   (2) $ i ,\, j\notin\tt;$  \, (3) $i,\,j \in \tt. $
\end{Cor}

\begin{Remark}\label{Upsilon}We denote three closed subsets of the root system
$\Phi$ of $G$ with respect to the three conditions above as follows:
  \begin{eqnarray*}\Upsilon_1&=&\{(i,j)\,|  \,i>j \text{ and }i \in \tt,\, j\notin\tt \},\\
  \Upsilon_2&=&\{  (i ,j ) \,|  \,i>j \text{ and }i , j\notin\tt \} ,\\
  \Upsilon_3&=&\{(i, j) \,|  \,i>j \text{ and }i ,j \in \tt \}.
  \end{eqnarray*}
 Then
  $\Upsilon=\Upsilon_1 \cup \Upsilon_2\cup \Upsilon_3$ is also a closed subset of  $\Phi.$ Thus $U^w\cap U=\prod_{(i,j)\in \Upsilon}X_{ij}$
  where the product can be taken in any order.  And the following statements follow easily by direct calculation:
\end{Remark}

\begin{Lemma} \label{ocr}Keep the notations of $\Upsilon_1, \Upsilon_2, \Upsilon_3$ as in Remark \ref{Upsilon}. Set
$$U^w_0=\{ \Pi\, x_{ij}(\alpha)\,|\,(i,j)\in \Upsilon_1 , \alpha\in \F_q\},$$
$$U^w_\c=\{ \Pi \,x_{ij}(\alpha)\,|\,(i,j)\in \Upsilon_2, \alpha\in \F_q\},$$
$$U^w_\r=\{\Pi \,x_{ij}(\alpha)\,|\,(i,j)\in \Upsilon_3, \alpha\in \F_q\}$$
Then $U^w_0$ is a normal subgroup of $U^w\cap U$ and $U^w\cap U=U^w_0\rtimes (U^w_\c\times U^w_\r).$
\end{Lemma}

\subsection{Monomial action of $U^w\cap U$ on $\mathcal E_\t$}

We now investigate the action of $U^w\cap U$ on \mbox{$\mathcal E_\t=\{e_{_L}\,|\,L \in \X_{\t}\}.$}
\begin{Prop}\label{mono}
 $U^w\cap U$ acts monomially on $\mathcal E_\t,$ that is given $L \in\X_{\t},$ \mbox{$g\in U^w\cap U,$} then there exist
 $K\in \X_{\t}$ and $0\not=C(L,g)\in F$ such that \mbox{$e_{_L}\circ g=C(L,g)e{_{_K}}.$}
\end{Prop}
\begin{proof}Note that
\begin{equation}\label{newexpan}e_{_L}\circ g=\frac{1}{q^{|\mathfrak{J}_{\t}|}}\sum_{\text{\tiny $M$}}\chi_{_L}(-M)[M\circ g]
=\frac{1}{q^{|\mathfrak{J}_{\t}|}}\sum_{\text{\tiny $M$}}\chi_{_L}(-M\circ g^{-1})[M]\end{equation} where $M$ runs through $\X_{\t}$, since then $M\circ g^{-1}$ runs through $\X_\t$ as well.
Keeping the notation in  \ref{Upsilon}, it is enough to prove the result for matrices of the form \mbox{$g=E_n+\alpha \epsilon_{_{ij}}$},
where $E_n$ is the $(n\times n)$-unit matrix, \mbox{$0\not= \alpha \in \F_q$}, $\epsilon_{_{ij}}$ is the $n\times n$-matrix unit to position
$(i,j)\in \Upsilon=\Upsilon_1 \cup \Upsilon_2\cup \Upsilon_3$.
So let \mbox{$g=E+\alpha \epsilon_{_{ij}}$}.

For $M\in \X_\t $ and $g^{-1}=E_n-\alpha \epsilon_{_{ij}}$, $Mg^{-1}$ is obtained from $M$ by adding $-\alpha$ times column $i$ to column $j$ of $M,$ therefore
$j\notin \tt$  implies that the columns of $M$ containing a last one are not changed by the action of $g$ and hence
$M\circ g^{-1}=Mg^{-1}$  for $(i,j)\in \Upsilon_1 \cup \Upsilon_2$.

\medskip

\noindent \textbf{Case (1):} $(i,j)\in \Upsilon_1$. That is  $i>j,\, i \in \tt$
and  $ j\notin\tt.$ Then
\begin{eqnarray*} &&\chi_{_L}(-M\circ g^{-1})=\chi_{_L}(-M g^{-1})
 = \prod\nolimits_{(u, v)\in \J_\t}\theta\big(-l_{uv}(Mg^{-1})_{uv}\big)
\\&=&\theta(-l_{ij}(m_{ij}-\alpha))\prod\nolimits_{(u, v)\not=(i,j)}\theta\big(-l_{uv}m_{uv}\big)=\theta(\alpha l_{ij})\chi_{_L}(-M).
\end{eqnarray*}
In this case
 $C(L,g)= \theta(\alpha l_{ij})$ and we have  $e_{_L}\circ g=C(L,g)e{_{_L}}.$

\medskip

\noindent \textbf{Case (2):} $(i,j)\in \Upsilon_2$. That is
$i>j$ and $i , j\notin\tt.$ Then 
\begin{eqnarray*} &&\chi_{_L}(-M\circ g^{-1})=\chi_{_L}(-M g^{-1})
= \prod\nolimits_{(u, v)\in \J_\t}\theta\big(-l_{uv}(Mg^{-1})_{uv}\big)
\\&=&\prod\limits_{(u, j)\in \J_\t}\theta\big(-l_{uj}(m_{uj}-\alpha m_{ui})\big)
\prod\limits_{(u, v)\in \J_\t \atop{v\not=j}}\theta\big(-l_{uv}m_{uv}\big)
\\&=&\prod\limits_{(u, i)\in \J_\t}\theta\big(-(l_{ui}-\alpha l_{uj}) m_{ui})\big)
\prod\limits_{(u, v)\in \J_\t \atop{v\not=i}}\theta\big(-l_{uv}m_{uv}\big)
=\chi_{_{K}}(-M)\end{eqnarray*}  where  $K\in \X_\t$ coincides with $Lg^{-t}$ in all positions in $\J_\t$.
In this case, we   have $e_{_L}\circ g=C(L,g)e{_{_{K}}}$ with $C(L,g)=1$.

\medskip

\noindent\textbf{Case (3):} $(i,j)\in \Upsilon_3$. That is   $i>j$ and $i , j\in\tt.$  Note in this case $Mg^{-1}\not=M\circ g^{-1}$ in general,  hence we need to row reduce $Mg^{-1}$ to obtain $M\circ g^{-1}.$ By easy calculation, we have
$\chi_{_L}(-M\circ g^{-1})=\chi_{_L}(-hMg^{-1})$
with $h=E_m+\alpha \tilde \epsilon_{_{ij}}$ where $E_m$ is the $(m\times m)$-unit
matrix, and $\tilde \epsilon_{_{ij}}$ is the $m\times m$-matrix unit to position
$(i,j)\in \Upsilon_3$. 
\begin{eqnarray*}&&\chi_{_L}(-M\circ g^{-1})=\chi_{_L}(-hMg^{-1})= \prod\nolimits_{(u, v)\in \J_\t}\theta\big(-l_{uv}(hMg^{-1})_{uv}\big)\\
&=&\prod\limits_{(u, v)\in \J_\t}\theta\big(-l_{uv}(h M)_{uv}\big)
=\prod\limits_{(i, v)\in \J_\t}\theta\big(-l_{iv}(m_{iv}+\alpha m_{jv})\big)\prod\limits_{(u, v)\in \J_\t\atop{u\not=i}}\theta\big(-l_{uv}m_{uv}\big)
\\&=&
\prod\limits_{(j, v)\in \J_\t}\theta\big(-(l_{jv}+\alpha l_{iv}) m_{jv}\big)\prod\limits_{(u, v)\in \J_\t\atop{u\not=j}}\theta\big(-l_{uv}m_{uv}\big)
=\chi_{_{K}}(-M)\end{eqnarray*}
where $K\in \X_\t$ coincides with $h^{t}L$ in all positions in $\J_\t$.
In this case, we   have $e_{_L}\circ g=C(L,g)e{_{_{K}}}$ with $C(L,g)=1$.
\end{proof}

\begin{Cor} \label{calcor}
We collect the information from the proof of the previous proposition as  follows: For $L\in \X_\t, g=E+\alpha \epsilon_{_{ij}}\in U^w\cap U$:
\begin{equation}\label{ucondition}e_{_L}\circ g=
\begin{cases}
\theta\,(l_{ij}\alpha)\,e_{_L}& \text{ if } i \in \tt,\, j\notin\tt;\\
e_{_K}& \text{ if  } i ,\,j \notin \tt;\\
e_{_R}& \text{ if  } i ,\,j \in \tt.
\end{cases}
\end{equation}
where $K=(k_{_{b_uv}})\in \X_\t , R=(r_{_{b_uv}})\in \X_\t$ satisfy:
\begin{equation}\label{column}k_{_{b_uv}}=
\begin{cases}
l_{_{b_uv}}& \text{ if  } v\not=i;\\
l_{_{b_ui}}-\alpha l_{_{b_uj}}& \text{ if  } v=i, i<b_u.
\end{cases}
\end{equation}
\begin{equation}\label{row}r_{_{b_uv}}=
\begin{cases}
l_{_{b_uv}}& \text{ if  } b_u\not=j;\\
l_{jv}+\alpha l_{iv}& \text{ if  } b_u=j, v<b_u.
\end{cases}
\end{equation}
From (\ref{column}) follows that  the action of $g=E+\alpha \epsilon_{_{ij}} \in U^w\cap U$ on $e_{_L}$ under the condition $i ,\,j \notin \tt$ is equivalent to subtracting in $L$ from the $i$-th column $\alpha$ times the  $j$-th column   ignoring the $(s,t)$-entries with $s\leqslant t$ and take the idempotent corresponding to the resulting matrix.  Hence we call this a \textbf{truncated  column operation}. Similarly, by (\ref{row}), the action of $g=E+\alpha \epsilon_{_{ij}}\in U^w\cap U$ on $e_{_L}$ under the condition $i ,\,j \in \tt$ is equivalent to adding $\alpha$ times the  $i$-th row to the $j$-th row of $L$  ignoring the $(s,t)$-entries with $s\leqslant t$ and take the idempotent corresponding to the resulting matrix. We call this a \textbf{truncated row operation}.
\end{Cor}

With respect to this monomial action, we can define $U^m\cap U$-orbit naturally:
For  $ e_{_{L}}\in\E_{\t}$ with $\t=\t^\lambda w$,
the $U^w\cap U$-orbit   of $ e_{_{L}}$ is
$$\O_{\!_L}= \{ e_{_{K}}\,|\,e_{_L}\circ g=C(L,g)e_{_K} \text{ for some } g\in U^w\cap U,\, 0\not=C(L,g)\in F \,\}.$$
and let $M_{\O_{\!_L}}=\bigoplus\limits_{e_{\!_K}\in \O_{\!_L}}Fe_{_K}$ be the  corresponding $U^w\cap U$-orbit module of $\M_\t$.

\subsection{The irreducibility of  the $U^w\cap U$-orbit module $M_\O$}

From the previous section, we know $\M_\t$ decomposes naturally into a direct sum of  $U^w\cap U$-submodules $M_\O, $
where $\O$ runs through the set of orbits of $U^w\cap U$ acting on $\E_\t.$
Our goal in this section is to classify the orbits $\O,$ determine their size (and hence the $F$-dimension of $M_\O$) and count the number of orbits of a given fixed size. We shall show that this number is a polynomial in $q$ with integral coefficients and the sizes of the orbits are powers of $q$; moreover for a given orbit $\O,$ the corresponding monomial $U^w\cap U$-module $M_\O$
is irreducible.

\begin{Defn}\label{defofomega}
For  $(b ,j)\in \J_\t,$ we define the hook $\h_{_{bj}}=\h_{_{bj}}^\t$ (of $\t$)  as:
$\h_{_{bj}}=\h_{_{bj}}^{^{\up}} \cup \h_{_{bj}}^{^{\ri}} \cup \{(b,j)\} \text{ where }$
$h_{_{bj}}^{^{\up}}=\{(u,j)\in \J_\t\,|\,u < b\} \text{ called {\bf hook leg}}$ and
$h_{_{bj}}^{^{\ri}}=\{(b,v)\in \J_\t\,|\,v > j\} \text{ called {\bf hook arm}}.$
Denote $\bar h_{_{bj}}=h_{_{bj}}^{^{\up}} \cup h_{_{bj}}^{^{\ri}}$ and call $|h_{_{bj}}|$ the residue of the hook, denoted by $\Res(b,j)$. In fact, it is easy to prove the following lemma:
\end{Defn}

\begin{Lemma}\label{res} For $(b,j)\in\J_\t$, $\Res(b,j)=b-j.$
In particular,  $\Res(b,j)$ is independent of $\t\in\RStd(\lambda)$ and independent of the two-part partition $\lambda.$
\end{Lemma}
We remark that this property of hooks is the deeper reason for labeling the rows of matrices in $\Xi_{m,n}$ in this unusual way.
This will allow us later on to compare orbits for different row standard tableaux even for different 2-part partitions.
\begin{Example} \label{omega}
Let  $\t=
\begin{tabular}{|c|c|c|}\hline1&2&4\\\hline 3&5&6\\\hline\end{tabular}\,, \s=
\begin{tabular}{|c|c|c|}\hline1&3&4\\\hline 2&5&6\\\hline\end{tabular}\,, \u=
\begin{tabular}{|c|c|c|c|}\hline1&3&4&6\\\hline 2&5 \\\cline{1-2}\end{tabular}\,.$ Then
$$
h_{_{51}}^\t\!=\!\!\raisebox{5pt}{$
\begin{matrix}\begin{matrix}\quad\emph 1&\,\emph 2&\,\emph 3&\,\emph 4&\emph 5&\emph 6&\end{matrix}\\
\begin{pmatrix} \times & &1 & &\\\times&\times&0&\times&1&\\ & &0& &0&1\end{pmatrix}&\begin{matrix}\!\!\!\!\!\!\emph 3\\\!\!\!\!\!\!\emph 5\\\!\!\!\!\!\!\emph 6\end{matrix}\,,\quad\end{matrix}$}
h_{_{51}}^\s\!=\!\!\raisebox{5pt}{$
\begin{matrix}\begin{matrix} \,\,\,\,\emph 1&\,\emph 2&\,\emph 3& \,\emph 4&\emph 5&\emph 6&\end{matrix}\\
\begin{pmatrix}  \times &1& & &\\\times&0&\times&\times&1&\\  &0&  & &0&1\end{pmatrix}&\begin{matrix}\!\!\!\!\!\!\emph 2\\\!\!\!\!\!\!\emph 5\\\!\!\!\!\!\!\emph 6\end{matrix}\,,\quad\end{matrix}$}
h_{_{51}}^\u\!=\!\!\raisebox{5pt}{$
\begin{matrix}\begin{matrix} \,\quad\emph 1& \emph 2& \,\emph 3& \, \emph 4&\emph 5&\!\emph 6&\end{matrix}\\
\begin{pmatrix}  \times &1& & &\\\times&0&\times&\times&1&\end{pmatrix}&\begin{matrix}\!\!\!\!\!\!\emph 2\\\!\!\!\!\!\!\emph 5\end{matrix}\,.\end{matrix}$}$$
and  $\Res(5,1)=4.$
\end{Example}

\begin{Defn} \label{patternmatrix}
\begin{itemize}
\item [(1)] A subset $\p=\{(b_i,a_i)\,|\,1\leqslant i \leqslant s\}\subseteq \Phi^+$, $0\leqslant s \leqslant n$ is called a {\bf pattern,} if the following holds:\vspace{-2mm}
   \begin{itemize}
      \item [1)]  $1\leqslant b_1<\cdots<b_s\leqslant n$.  \vspace{-1mm}
      \item [2)]  $a_1,\cdots,a_s\in \{1,\dots,n\}\setminus \{b_1,\dots,b_s\}$ are pairwise different. \vspace{-1mm}
      \item [3)]  $a_i<b_i$ for $i=1,\dots,s$. \vspace{-1mm}
   \end{itemize}
\vspace{-1mm}
For $\lambda=(n-m,m)\vdash n$, $\t\in \RStd(\lambda)$, we say pattern $\p$ {\bf fits} the $\t$-batch $\M_\t$ and we call $\p$ a 
{\bf $\lambda$-pattern}, if $(b_1,\dots,b_s)\subseteq \tt$. Thus $\p$ is a $\lambda$-pattern if and only if $s\leqslant m$. \vspace{-2mm}
\item[(2)]  
 $L=(l_{{ij}})\in \Xi_{m,n}$ is called a \textbf{pattern matrix}, if each row and column of $L$ has at most one non zero entry besides the last 1's. The corresponding
idempotent $e_{_L}$ is called {\bf pattern idempotent}; it is easy to see for a pattern matrix $L$, the set of positions $(i,j)\in \J_\t$ with $l_{_{ij}}\not=0, \t=\tab(L)$ satisfies the condition for pattern in (1), thus we call it     
{\bf pattern of $L$}, denoted by $\p=\p(L)$.
We call their concrete values in $\F_q^*$ 
{\bf  a  filling of $\p$}, denoted by $\pf(L)$.
\end{itemize}
\end{Defn}


\begin{Remarks}\label{outcondition} 
\begin{itemize}
\item [(1)]For $\p \subseteq \Phi^+$ a pattern and $\lambda=(n-m,m)\vdash n$, there exists pattern matrices and pattern idempotents in $M^\lambda$ with associated pattern $\p$ if and only if $|\p|\leqslant m$. As we will see orbit modules of $U^w\cap U$  acting on $\E_\t$, $\t=\t^\lambda w$, are invariant under the action of $U$ and filled patterns are important invariants of these. 
This is the reason behind labeling the rows of matrices in $\Xi_{m,n}$ in this particular way.\vspace{-2mm}
\item [(2)]Since the truncated  column and row operations of $U^w\cap U$ on $e_{_L}$ work from left to right and  down to up, they will not insert any nonzero values to the southwest positions of  the outer rim of $\p(L)$. More precisely,  say $(b, j)\in {\po(L)}$, the outer rim of $\p(L)$, if the following holds:
If $(c, k)\in \p(L)$ and $k<j,$ then $c<b.$ Naturally
we  define  $\pfo(L)$ by taking  the concrete values together
with those indices in $\po(L)$. Obviously, for any $e_{_K}\in \O_L,$ we have: $k_{bj}=l_{bj}$ for $(b,j)\in \po(L).$
\end{itemize}
\end{Remarks}

Next we show that each orbit $\O$ of $\E_\t$ under the monomial action of $U^w\cap U$ contains precisely one pattern matrix and that the dimension of $M_\O$ is determined combinatorially by the frame of the corresponding filled pattern.

\begin{Lemma} Each $ U^w\cap U$-orbit $\O$ of $\E_\t$ contains  a unique pattern idempotent. So we have a bijection between the $ U^w\cap U$-orbits of $\E_\t$   and pattern matrices in $\X_\t$. Moreover,
for a fixed pattern $\p$ there are precisely $(q-1)^{s}$ many different pattern matrices $L$ and orbits $\O_L$ such that $\p(L)=\p$, where $s=|\p|$, the cardinality  of $\p$.\end{Lemma}
\begin{proof}
First we prove the existence.
Let $e_{_K}\in \E_\t\subset \O$. Assume $j$ is the first column of $K$ containing nonzero values besides the last 1. Choose the lowest nonzero value in this column, namely $k_{{bj}}$.
Using truncated row and column operations  we can obtain a matrix $M=(m_{cd})\in \X_\t$ with $e_{_M}\in \O$ such that $m_{{bj}}=k_{{bj}}$ and all entries $m_{cd}$ with $(c,d)\in \bar h_{_{bj}}$  are zeros. Then we go to the next column which contains nonzero values besides the last 1 and do the same procedure. Continuing in this way, we will finally obtain a pattern
matrix $L$ such that $e_{_L}\in \O.$

Next we show the uniqueness. Suppose we have two different pattern idempotents $e_{_L}, e_{_R} \in \O$  with respectively filled pattern $\pf(L)$ and $\pf(R)$. Assume $e_{\!_R}=e_{\!_L}g$ for some $g\in U^w\cap U.$ Using \ref{Upsilon}, we can assume $g=g_{_1}g_{_2}g_{_3}$ where
$e_{\!_L}g_{_1}=C(L,g_{_1})e_{\!_L}$ and $g_{_2}$ is a series of products of truncated column operations and $g_{_3}$ is a series of products of truncated row operations. Now we have \begin{equation}\label{uniquep}e_{\!_L}\circ g_{_1}g_{_2}=C(L,g_{_1})e_{\!_L}\circ g_{_2}=e_{\!_R}\circ g_{_3}^{-1},\end{equation} where $g_{_3}^{-1}$ is again  a series of products of truncated row operations.
From  \ref{calcor} and   \ref{outcondition}, we can easily get $\pfo(L)=\pfo(R)$. If we order the condition sets $\pf(L)$ and $\pf(R)$ by the column indices, then we can choose without lose generality  the first $l_{_{uv}}\in \pf(L)$ and $l_{_{uv}}\notin \pf(R),$ such that $r_{st}\in \pf(R)\setminus \pfo(R)$ has the property $t\geqslant v.$
For $t>v$ or $t=v, s<u:$
Since $L$ is a pattern matrix, (\ref{column}) shows that the truncated column operations only change the hook row on the $(i,j)$-hook with $l_{ij}\in \pf(L).$ Hence we get $m_{_{uv}}=l_{_{uv}}$ for any $e_{_M}=e_{\!_L}\circ g_{_2}.$
Similarly since $R$ is a pattern matrix, the truncated row operations only change the hook column on the $(s,t)$-hook with $r_{st}\in \pf(R).$ Hence we get $m_{_{uv}}=0\not=l_{_{uv}},$ which means (\ref{uniquep}) never holds. Therefore,  $\pf(L)=\pf(R)$ and hence $e_{_L}=e_{_R}.$ For $t=v, s>u$, considering the position $(s,v)$ instead of $(u,v)$ we will get the same result similarly, which proves the uniqueness.
\end{proof}

Since we have proved each orbit has a unique pattern matrix, we can now define $\tab(\O)=\tab(L)$,   $\p(\O)=\p(L)$, and  $\pf(\O)=\pf(L).$

\begin{Notation}\label{s_i} Let $L$  be a pattern matrix with pattern  $\p=\p(L)$. Define:\vspace{-2mm}
\begin{itemize}
\item [(1)] $\ppi =\{i\,|\,{(i,j)\in   \p \,}\},$  which collects all the row indices of the positions in   $\p$. \vspace{-2mm}
\item [(2)] $\pj=\{j\,|\,{(i,j)\in    \p \,}\},$ which collects all the column indices of the positions in   $\p$.\vspace{-2mm}
\end{itemize}
Note that $ \p_{_{\mathcal I}}=\p_{_{\mathcal J}}=\emptyset$ if and only if $\p=\emptyset$.
\end{Notation}

Now we try to determine the size of an $U^w\cap U$-orbit. Since $U^w\cap U=U^w_0\rtimes (U^w_\c\times U^w_\r),$
using Corollary \ref{calcor} we see that every element in $U^w_0$ is  in the projective stabilizer of $e_{_L},$ hence in
order to compute the orbits of the action of $U^w\cap U$, it suffices to calculate  $\Stab_{\mbox{\tiny $U^w_\c\times U^w_\r$}}(e_{_L}),$ for a pattern idempotent $e_{_L}$, since by (\ref{column}) and (\ref{row}), the projective stabilizer of $e_{_L}$ in $U^w_\c\times U^w_\r$ is exactly the stabilizer of $e_{_L}$ in it.

\begin{Lemma}\label{stabcr}
Let $\O$ be an $U^w \cap U$-orbit with pattern idempotent $e_{_L}$ and  pattern $\p.$ Then
$\Stab_{\mbox{\tiny $U^w_\c $}}(e_{_L})=\langle X_{_{ij}}\,|\, i,j\notin\tt\,;\, j\notin \pj\text{ or }\exists\, (b,j)\in \p \text{ with } b<i\,\rangle$ and
$\Stab_{\mbox{\tiny $U^w_\r $}}(e_{_L})=\langle X_{_{ij}}\,|\, i,j\in\tt\,;\, i\notin \ppi\text{ or }\exists\, (i,v)\in \p \text{ with } v>j\,\rangle.$
\end{Lemma}
\begin{proof}
From (\ref{column}) the truncated column action on  $e_{_L}$ induced by $x_{_{ij}}(\alpha)$ is just subtracting in $L$ from the $i$-th column $\alpha$ times the  $j$-th column ignoring in column $i$ all zero entries to the right
of a last one and taking the idempotent indexed by the resulting matrix. Hence $X_{{ij}}\in \Stab_{\mbox{\tiny $U^w_\c $}}(e_{_L})$ if and only if   the $j$-th column of  the pattern matrix $L$ is a zero column, that is   $j\notin\pj,$ or there exists $(b,j)\in \p$ with $b<i$.
The calculation of $\Stab_{\mbox{\tiny $U^w_\r $}}(e_{_L})$ is carried out similarly.
\end{proof}

\begin{Prop}\label{dimoforbit}Let $ \O$ be a $\,U^w\cap U$ orbit and assume
$\p=\p(\O)=\{(b_{u_i}, v_i)\,|\,1\leqslant i \leqslant s\}.$ Thus $s=|\p|$. Then  $\dim M_\O= q^{k-s},$ where $k$ is the number of places which are on  the hooks whose corners belong to the pattern $\p$. More precisely,
$k=\sum_{1\leqslant i \leqslant s}\big( (b_{u_i}-v_i)-|Z_i|\big)$
where $Z_i=\{j\,|\,b_{u_j}>b_{u_i}>v_j>v_i\,\}$ for $1\leqslant i \leqslant s.$
\end{Prop}
\begin{proof}Let $e_{_L}$ be the unique pattern idempotent in $\O,$ and let $\pf=\pf(L).$ Now we calculate the stabilizer of $e_{_L}$ in $U^w\cap U$.
We have already got three types of projective stabilizer of $e_{_L}$ by Corollary \ref{calcor} and Lemma \ref{stabcr}:
\begin{itemize}
\item[(1)] $U^w_0=\langle X_{_{ij}}\,|\, i \in \tt,\, j\notin\tt \,\rangle;$ \vspace{-2mm}
\item[(2)] $\Stab_{\mbox{\tiny $U^w_\c $}}(e_{_L})=\langle X_{_{ij}}\,|\, i , j\notin\tt\, ; j\notin \pj\text{ or }\exists\, (b,j)\in \p \text{ with } b<i\rangle;$ \vspace{-2mm}
\item[(3)] $\Stab_{\mbox{\tiny $U^w_\r $}}(e_{_L})=\langle X_{_{ij}}\,|\, i  , j\in\tt\, ;\, i\notin \ppi\text{ or }\exists\, (i,v)\in \p \text{ with } v>j\rangle.$
\end{itemize}
Moreover, since we may have some intersection positions which are both on some hook row and some hook column with those hooks whose corners belonging to the pattern  $\p$ and again by Corollary \ref{calcor} there exist pairs of row operations and column operations such that the product of these two operations acts trivially on $e_{_L}.$ More precisely,
the pair has the form $x_{{ij}}(\alpha_{{ij}})x_{{st}}(\beta_{{st}})$ with $\alpha_{{ij}}l_{{tj}}=\beta_{{st}} l_{{si}}$ where $l_{{si}},l{_{tj}}\in \pf$ and $t>i$. That means
$$\mathbf P:=\{x_{{ij}}(\alpha_{{ij}})x_{{st}}(\beta_{{st}})\,|\,\alpha_{{ij}} \in \F_q, t>i, l_{si},l_{tj}\in \pf,\alpha_{{ij}}l_{{tj}}=\beta_{{st}} l_{{si}} \}$$ is a set  of some elements contained in the stabilizer of $e_{_L}$.

By Lemma \ref{Upsilon}, $U^w\cap U=\prod_{(i,j)\in \Upsilon}X_{ij}$
  can be taken in any order. Hence for  any $g=\prod x_{_{ij}}(\alpha_{_{ij}})\in U^w\cap U,$ we can
fix an order like this: firstly
those $x_{_{ij}}(\alpha_{_{ij}})$ belonging to $U^w_0,\, \Stab_{\mbox{\tiny $U^w_\c $}}(e_{_L})$ and
$\Stab_{\mbox{\tiny $U^w_\r $}}(e_{_L})$, then those pairs $x_{_{ij}}(\alpha_{_{ij}})x_{_{st}}(\beta_{_{st}})$ belonging to the set $\mathbf P$, then the remaining truncated column operation, and at last the remaining truncated row operation. Since for those truncated row operation $x_{_{st}}$ in the pair set, there is a uniquely expression $x_{_{st}}(\gamma_{_{st}})=x_{_{st}}(\beta_{_{st}})\,x_{_{st}}(\gamma_{_{st}}-
\beta_{_{st}})$ for any $\gamma_{_{st}}\in \F_q,$  this order makes sense.

Suppose $u_1$ is the product of the remaining truncated column operation of $g$,  $u_2$ is  the product of the remaining truncated row operation of $g$, then
$$u_1\in \prod X_{ij} \text{ with }
(i,j)\in \Gamma_1:=\{(i,j)\in \Upsilon_2\,|\,i\notin \pj, (b,j)\in \p \text{ for some } b>i\}$$
$$u_2\in \prod X_{st}\text{ with  }(s,t)\in\Gamma_2:=\{(s,t)\in \Upsilon_3\,|\, (s,v)\in \p \text{ for some }v<t\}.$$
Note that by (\ref{ucondition}), $u_1u_2\in \Stab_{\mbox{\tiny $U^w_\c\times U^w_\r$}}(e_{_L})$ iff $e_{_L}\circ u_1u_2=e_{_L}.$
We claim that: $e_{_L}\circ u_1 u_2=e_{_L}$ if and only if $u_1=u_2=1.$

Now
$e_{_L}\circ u_1 u_2=e_{_L} \Leftrightarrow e_{_L}\circ u_1 =e_{_L}\circ u_2^{-1}$ where $u_2^{-1}$ is again a truncated row operation and  belongs to $\prod X_{st}\text{ where  }s,t\in\tt, (s,v)\in \p$, for some $v<t$.
Moreover, by (\ref{column}), $e_{_L}\circ u_1$ has only possible nonzero entries on the positions in  rows $u\in \ppi$ except those whose column indices belonging to $\pj.$
And by (\ref{row}), $e_{_L}\circ u_2^{-1}$ has only possible nonzero entries on the positions in  columns $v\in \pj.$
It means that the action of $u_1$ and $u_2^{-1}$ on $e_{_L}$  influence different positions. Hence
\begin{equation}\label{rest}e_{_L}\circ u_1 =e_{_L}\circ u_2^{-1}\Leftrightarrow u_1=u_2^{-1}=1\Leftrightarrow u_1=u_2=1.\end{equation}
Therefore $e_{_L}\circ g =C(L,g)e_{_L}$ implies $g=g_{_1}g_{_2} $ where $g_{_1}\in U^w_0 \cup \Stab_{\mbox{\tiny $U^w_\c $}}(e_{_L})\cup
\Stab_{\mbox{\tiny $U^w_\r $}}(e_{_L})$ and $g_{_2}\in \mathbf P.$

Now  for  $g=g_{_1}u_{_1}u_{_2},\, h=h_{_1}v_{_1}v_{_2}$ with
\mbox{$g_{_1},h_{_1}\in$ $U^w_0 \cdot \Stab_{\mbox{\tiny $U^w_\c $}}\cdot \Stab_{\mbox{\tiny $U^w_\r $}}\cdot \mathbf P$} and
$u_{_1}$(resp. $u_{_2}$) is  the product of the remaining truncated column (resp. row) operation of $g$,  $v_{_1}$(resp. $v_{_2}$) is  the product of the remaining truncated column (resp. row) operation of $h.$
We claim: $e_{_L}\circ u_1u_2=e_{_L}\circ v_1v_2$ if and only if $u_1=v_1, u_2=v_2.$

Since   truncated row and column operations commute  with each other, we get
$e_{_L}\circ u_1u_2=e_{_L}\circ v_1v_2\Leftrightarrow e_{_L}\circ u_1u_2=e_{_L}\circ v_2 v_1\Leftrightarrow
e_{_L}\circ u_1v_1^{-1}=e_{_L}\circ v_2 u_2^{-1}.$
Using   (\ref{rest}), we obtain
$ e_{_L}\circ u_1v_1^{-1}=e_{_L}\circ v_2 u_2^{-1} \Leftrightarrow u_1v_1^{-1}=v_2 u_2^{-1}=1\Leftrightarrow u_1=v_1, u_2=v_2.$
Hence $e_{_L}\circ u_1u_2$ gives all the  coset representatives of $\Stab_{\mbox{\tiny $U^w\cap U$}}(e_{_L})$ in $U^w\cap U$ where
$u_1\in \prod X_{ij}$ with $i,j\notin\tt,$ $i\notin \pj,$ $(b,j)\in \p$ for some $b>i$ and $u_2\in \prod X_{st}\text{ where  }s,t\in\tt, (s,v)\in \p$ for some $v<t$.

Now we can calculate the size of the orbit, which is just the index of the projective stabilizer in $U^w\cap U$, namely $q^d$ where $d={|\Gamma_1|+|\Gamma_2|}$. For \mbox{$1\leqslant i \leqslant s=|\p|,$} set
$Z_i=\{j\,|\,b_{u_j}>b_{u_i}> v_j>v_i\,\}$ denoting the number of the hook intersections  on $b_{u_i}$-th row with hooks centered at positions in the pattern $\p$.
Moreover, if we denote
\mbox{$\tilde\Gamma_1:=\{(i,j)\in \Upsilon_2\,|\,(b,j)\in \p \text{ for some } b>i\}$} then obviously, $|\Gamma_1|+|\Gamma_2|=\tilde\Gamma_1+|\Gamma_2|-\sum_{i=1}^s |Z_i|$.

It is easy to see $|\tilde \Gamma_1|$ is the number of   all the positions on the hook arm with corners     in $\p$,
and respectively $|\Gamma_2|$ is the the number of     all the positions on the hook column places with corners     in $\p$.
Then by   \ref{res},
$d=\sum_{1\leqslant i \leqslant s}  (b_{u_i}-v_i)-s-\sum_{i=1}^s |Z_i|
=\sum_{1\leqslant i \leqslant s} \big( (b_{u_i}-v_i)-s-|Z_i|\big).$
Let $k=\sum_{1\leqslant i \leqslant s} \big( (b_{u_i}-v_i)-|Z_i|\big),$ then  it is  the number of places which are on   the hooks whose corners belong to the pattern $\p$. Therefore we obtain $\dim{M_\O}=q^ {k-s}. $
\end{proof}

\begin{Remark}\label{nodependency}By    \ref{dimoforbit}, if two orbits have the same pattern  then they have the same dimension. Even in the case that the tableaux of   two orbits having different shapes, the statement still holds. Therefore, for a given frame of a filled pattern, the number of all the admissible orbits   is a polynomial in $q$ with integral coefficients and the sizes of the orbits are   powers of $q.$
Moreover, if the elements on the hooks with corners belonging to the pattern are fixed, then we have no choice of the other places, otherwise the dimension of the orbit will be increased.\end{Remark}

\begin{Example}\label{short for orbit}
Let $L=\begin{pmatrix}
0 & 0 & 1 & 0\\
z & 0 & 0 & 1
\end{pmatrix}$ with $0\not=z\in \F_q$. Then\\ \mbox{$\O_L=\left\{e_{_K}\,\Big|\,K=\begin{pmatrix}
a & 0+\square & 1 & 0\\
z & b & 0 & 1
\end{pmatrix}\text{ where } \square \text{ is determined by } a,b\in \F_q\right\}$
and $\dim M_{\O_L}=q^2.$} In particular, we introduce a short notation for the orbit module $M_{\O_L}$ for later use:
$$M_{\O_L}= \begin{pmatrix}* & 0+\square & 1 & 0\\
z & * & 0 & 1
\end{pmatrix}\, .$$
\end{Example}

\smallskip

Next we shall prove the irreducibility of   the $(U^w\cap U)$-orbit modules:
\begin{Theorem}\label{irr}
Let $\lambda=(n-m,m)\vdash n,\, \O$ be an $\,U^w\cap U$-orbit with \mbox{$ \tab(\O)=\t^\lambda w.$} Then $M_{\mathcal{O}}$
is an irreducible $F(U^w\cap U)$-submodule of the batch $\M_\t$.
\end{Theorem}
\begin{proof} Let $e_{_L} $ be the unique pattern idempotent in $\O$.
We need to prove: For any arbitrary element $x\in M_\O= e_{_L}  F(U^w\cap U)$, we have  $xF(U^w\cap U)=M_{\mathcal{O}}$. Write $x=\sum_{e_{_K}\in \O} a_{_K}e_{_K}.$
Since $U^w\cap U$ acts monomially on $\O$, we can reduce our problem to the simple case:
$x=e_{_L}+\sum_{K\neq L}a_{_K}e_{_K}.$

Let $\p=\p(L)$ and
 $\Omega =\bigcup _{(b , j )\in \p}\bar h_{_{bj }}^\t.$ Note that for $(b,v)\in \Omega$ we have automatically $b\in\tt,$ \mbox{$v\notin\tt$} and by Corollary \ref{calcor} $X_{bv}\in U^w_0$.
More precisely, for any $e_{_R}\in \O$ and $ x_{bv}(\alpha)=E+\alpha  e_{_{bv}}$ with $(b,v)\in \Omega, \,\alpha\in \F_q$ we have $e_{_R}\circ x_{bv}(\alpha)=
\theta\,(r_{bv}\alpha )\,e_{_R}.$ Set
$a=\prod\nolimits_{(b,v)\in \Omega}\sum\nolimits_{\alpha\in   \F_q} x_{bv}(\alpha).$
Then
$e_{\!_{L}}\circ a=
\prod_{(b,v)\in \Omega}\sum_{\alpha \in      \F_q}\theta\,(l_{_{bv}}\alpha )\, e_{\!_{L}}=q^{|\Omega|} e_{\!_L}$, since $l_{{bv}}=0$ for all $(b,v)\in \Omega$. Moreover
$e_{\!_{K}}\circ a= \prod_{(b,v)\in \Omega}\sum_{\alpha \in  \F_q}
\theta\,(k_{_{bv}}\alpha )\, e_{\!_{K}}=0,$
since for any $K\not=L$, there exists at least one position $(b,v)\in \Omega$ such that $k_{_{bv}}\not=0$ and therefore the orthogonality relations for irreducible character of the group $(\F_q,+)$ imply
\begin{equation}\label{orth}\sum\nolimits_{\alpha \in  \F_q}
\theta\,(k_{_{bv}}\alpha )=\sum\nolimits_{\alpha\in \F_q}\theta(\alpha)=0\quad \text{  for $k_{_{bv}}\not=0$}.\end{equation}
Hence we obtain $ x \circ a=\big(e_{\!_L}+\sum_{K\neq L}a_{\!_K}e_{\!_K}\big) \circ a=e_{\!_L}\circ a= q^{|\Omega|} e_{\!_L}.$ This shows \mbox{$ x  F(U^w\cap U)=e_{\!_L}  F(U^w\cap U)=M_{\mathcal{O}}.$}
Therefore $M_{\mathcal{O}}$ is an irreducible \mbox{$F(U^w\cap U)$}-submodule of $\M_\t$ with $\t=\tab(\O)=\t^\lambda w.$
\end{proof}


\subsection{$U$-invariance of $M_\O$}
In this section, we fix a $\,U^w\cap U$-orbit $\O$ of $\E_\t$ with $\t=\t^\lambda w\in \RStd(\lambda),$ \mbox{$\lambda=(n-m, m)\vdash n. $} Let $\tt=(b_1,\cdots, b_m)$ and let $e_{_L}$ be the unique pattern idempotent in $\O.$ Now we show that the corresponding
module $M_\O$ is invariant under the action of $U$, which shows that $M_\O$ is actually an irreducible $FU$-module. We first show that in order to prove that $M_\O$ is invariant under the action of $U$, it suffices to prove $e_{_L}\circ g\in M_\O.$

\begin{Lemma}\label{patterninvariant}
If  $e_{_L}\circ g\in M_\O,\,\forall\, g\in U$. Then $M_\O$ is invariant under the action of $U.$
\end{Lemma}
\begin{proof}
By Theorem \ref{Upsilon}, we can define a normal sequence:
$$U^w\cap U=U_0\leqslant U_1 \leqslant \cdots \leqslant U_i \leqslant U_{i+1} \leqslant \cdots \leqslant U_k=U$$ such that
$U_i\trianglelefteq U_{i+1},$    for each $0\leqslant i \leqslant k-1.$

Now suppose that $e_{_L}\circ g\in M_\O$ for all $g\in U$ and suppose inductively that we have already shown that $M_\O$ is $U_i$-invariant for some $0\leqslant i \leqslant k-1$. We show that $M_\O$ is $U_{i+1}$-invariant, the case $i=0$ being trivial. Then the claim follows by induction.
The fact that $M_\O$ is $U^w\cap U$ invariant is our induction basis.
Inductively, suppose that $M_\O$ is $U_i$ invariant.
Let $g_{_1},\ldots,g_{_r}$ be generators of $U_{i+1}.$
Suppose $e_{_L}\circ g_{_j}\in M_\O\text{ for   }j=1,\ldots, r.$
Choose an arbitrary   $e_{_K}\in  \O$. Then there exist some $u\in U^w\cap U \leqslant U_i$ such that $e_{_L}\circ u=e_{_K}.$
Hence \mbox{$e_{_K}\circ g_{_j}=e_{_L}\circ u g_{_j}=e_{_L}\circ g_{_j} g_{_j}^{-1}u g_{_j}=(e_{_L}\circ g_{_j})\circ ( g_{_j}^{-1}u g_{_j}) \in M_\O $} since $U_i\trianglelefteq U_{i+1}$ and hence $g_{_j}^{-1}u g_{_j}\in U_{i}$. Then $ e_{_L}\circ g_{_j}\in M_\O$ since $M_\O$ is $U_i$ invariant by assumption. This shows that $M_\O$ is $U_{i+1}$ invariant.
\end{proof}

So we only need to prove:\vspace{-1mm}
\begin{Prop}\label{uinvariant}
$e_{_L}\circ g\in M_\O,\,\forall\, g\in U$. In particular, $M_\O$ is an irreducible $FU$-module.
\end{Prop}
\begin{proof}
It suffices to prove
$e_{_L}\circ g\in M_\O $
for $g=x_{_{ij}}(\alpha)\notin U^w\cap U,$ $\alpha\in\F_q$, (i.e. $i\notin \tt,\, j\in \tt,\, i>j$), since we know $M_\O$ is $U^w\cap U$-invariant.

Let $g=x_{_{vb_t}}(\alpha)\in U$ with  $v\notin \tt\,, b_t \in \tt$.
There exists  $1\leqslant s \leqslant m$ such that $ b_{s-1}<v<b_s$ then $1\leqslant t <s \leqslant m.$ By (\ref{newexpan}),
\begin{equation}\label{eee}e_{_L}\circ g
=\frac{1}{q^{|\mathfrak{J}_{\t}|}}\sum\nolimits_{\text{\tiny $M\in \X_\t$}}\chi_{_L}(-M\circ g^{-1})[M].\end{equation} Since $M g^{-1}$ is obtained from $M$ by replacing the zero entries on positions $(b_r, b_t)$ by $-\alpha m_{b_r v}$ for all $s\leqslant r \leqslant m$.
Then $M\circ g^{-1}$ is obtained from $Mg^{-1}$ by adding
$\alpha m_{b_r v}$ times row $b_t$ to row $b_r$ for all $s\leqslant r \leqslant m$.
That is $M\circ g^{-1}=hMg^{-1}$ with $h=\prod_{r=s}^m(E_m+\alpha m_{b_rv})$.
Similarly as the third case  in the proof of \ref{mono}, we obtain
$\chi_{_L}(-M\circ g^{-1})=\chi_{\hat L}(-M)$ where $\hat l_{ij}=(h^t L)_{ij}$
for $(i,j)\in \J_\t$.
More precisely, $\hat L$ is obtained from $L$ by replacing the entries on all positions $(b_t,j)\in \J_\t$ by $(l_{b_tj}+\sum_{r=s}^m \alpha m_{b_rv}l_{b_r j})$. Therefore
\begin{eqnarray}
&&\chi_{_L}(-M\circ g^{-1}) = \prod_{\stackrel{(b_i,j)\in \J_\t}{i\not=t}}\theta(-l_{b_ij}m_{b_ij})\prod_{(b_t,j)\in \J_\t}\theta\big(-(l_{b_tj}+\sum_{r=s}^m \alpha m_{b_rv}l_{b_r j})m_{b_tj}\big)\nonumber\\&=&
\prod\limits_{ (b_i,j)\in \J_\t}\theta(-l_{b_ij}m_{b_ij})\prod\limits_{(b_t,j)\in \J_\t}\theta\big(-\sum_{r=s}^m \alpha l_{b_r j}m_{b_rv}m_{b_tj}\big)\label{elg extension}\\\label{furtherexpand}
&=&\prod_{i \not= t,\, j \not=v \atop {(b_i,j)\in \J_\t}}\theta(-l_{b_ij}m_{b_ij}) \cdot\!\!\! \prod_{r=s\atop {(b_r,v)\in \J_\t}}^m\theta(-l_{b_rv}m_{b_rv})\cdot \prod_{(b_t,j)\in \J_\t}\theta\big(\!-(l_{b_tj}+\sum_{r=s}^m \alpha m_{b_rv}l_{b_r j})m_{b_tj}\big).\quad\quad\quad\,\end{eqnarray}
In (\ref{elg extension}), the critical term is the second product factor, since it contains a multiplication of $m_{b_rv}$ and $m_{b_tj}$. Obviously the easy case is when this critical term disappears, that is $l_{b_r j}=0$, for all $(b_r,j)\in \J_\t$ with $s\leqslant r \leqslant m$. In particular we get then $\chi_{_L}(-M\circ g^{-1})=\chi_{_L}(-M)$ and by (\ref{eee}) $e_{_L}\circ g=e_{_L}\text{  if  } l_{b_r j}=0 , \text{ for all } s\leqslant r \leqslant m,  1\leqslant j < b_t.$

Next we deal with the critical case, that is we have an $(b_r,j)\in \p=\p(L)\subset \J_\t$ for some $s\leqslant r \leqslant m, 1\leqslant j <b_t.$ Using general character theory we may
rewrite
$[M]=\sum _{K\in \X_{\t}} \chi_{_K}(M)e_{_K}$ in (\ref{eee}):
\begin{eqnarray}e_{_L}\circ g
&=&\frac{1}{q^{|\mathfrak{J}_{\t}|}}\sum_{\text{\tiny $M$}}\chi_{_L}(-M\circ g^{-1})\sum\limits_{K\in \X_{\t}} \chi_{_K}(M)e_{_K}\nonumber\\
&=&\sum\limits_{K\in \X_{\t}}\left(\frac{1}{q^{|\mathfrak{J}_{\t}|}}\sum_{\text{\tiny $M$}}\chi_{_L}(-M\circ g^{-1})\chi_{_K} (M)\right)e_{_K}.
\end{eqnarray}
Let $C_{_K}=\frac{1}{q^{|\mathfrak{J}_{\t}|}}\sum_{ M\in \X_\t}\,\chi_{_L}(-M\circ g^{-1})\chi_{_K} (M)$, then $e_{_L}\circ g=\sum _{K\in \X_{\t}} C_{_K}e_{_K}.$ Our strategy will be, to determine which $e_{_K}$ occurs with non zero coefficient in this expression. By (\ref{furtherexpand}),
\begin{eqnarray}
C_{_K}&=&\frac{1}{q^{|\mathfrak{J}_{\t}|}}\sum_{\text{\tiny $M\in\X_\t$}}\prod_{i \not= t,\, j \not=v \atop {(b_i,j)\in \J_\t}}\theta\big((k_{b_ij}-l_{b_ij})m_{b_ij}\big) \cdot \prod_{r=s\atop {(b_r,v)\in \J_\t}}^m\theta\big((k_{b_rv}-l_{b_rv})m_{b_rv}\big)\nonumber\\&&\cdot \prod_{(b_t,j)\in \J_\t}\theta\big((k_{b_tj}-l_{b_tj}-\sum_{r=s}^m \alpha m_{b_rv}l_{b_r j})m_{b_tj}\big)\label{ck1}
\end{eqnarray}
It is easy to see that the product $ \prod\limits_{i \not= t,\, j \not=v \atop {(b_i,j)\in \J_\t}} \sum_{m_{_{b_ij}}\in \F_q}\theta\,\((k_{_{b_ij}}-l_{_{b_ij}})m_{_{b_ij}}\)$ is a factor of $C_K.$
 Hence by (\ref{orth}) in order to get $C_K\not=0,$ all the terms $(k_{_{b_ij}}-l_{_{b_ij}})$ must be zero, which leads to
 $k_{_{b_ij}}=l_{_{b_ij}} \text{ for }   i \not= t,  j \not=v, (b_i,j)\in \J_\t.$ That means, if $C_K\not=0$ then $K$ and $L$ coincide in all positions except possibly ones in row $b_t$ or in  column $v.$ We remark that up to now, we have not used   the condition that $e_{_L}$ is a pattern idempotent. Now (\ref{ck1}) becomes:
\begin{eqnarray}
\!\!\!\!\!\!C_{_K}&\!\!\!\!\!=\!\!\!\!\!&\frac{1}{q^{a}}  \prod_{r=s\atop {(b_r,v)\in \J_\t}}^m\sum_{ m_{_{b_rv}}\in \F_q }\bigg(\theta\big((k_{b_rv}-l_{b_rv})m_{b_rv}\big) \cdot\nonumber\\&&\!\!\! \prod_{(b_t,j)\in \J_\t}\sum_{m_{_{b_t j}}\in \F_q}\theta\big((k_{b_tj}-l_{b_tj})m_{b_tj}-\sum_{r=s}^m \alpha l_{b_r j}m_{b_rv}m_{b_tj}\big)\bigg)\label{ck2}
\end{eqnarray}
where $a=\sharp\{(b_r,v), (b_t,j)\in \J_\t\,|\,s\leqslant r \leqslant m,1\leqslant j <b_t\}.$

Let $Y=\{(b_{u_i}, w_i)\in \p\,|\,s\leqslant u_i \leqslant m, 1\leqslant w_i< b_t, 1\leqslant i \leqslant \ell\}.$
The critical term $m_{b_rv}m_{b_tj}$ appears only when $(b_r,j)\in Y$, since otherwise $l_{b_r j}=0$. Therefore again by (\ref{orth}),   $C_K\not=0$ implies:
$$k_{b_rv}=l_{b_rv}\text{ for }r\not=u_i, \,\forall\, 1\leqslant i \leqslant \ell; \quad  k_{b_tj}=l_{b_tj}\text{ for }j\not=w_i,\,\forall\,1\leqslant i \leqslant \ell.$$
This shows the nonzero entries of $K$ only appear on a column or a row containing a position in   $Y$. In this sense, (\ref{ck2}) becomes:
\begin{eqnarray*}
C_{_K}&=&\frac{1}{q^{2\ell}}  \prod_{i=1\atop {(b_{u_i},v)\in \J_\t}}^l\sum_{ m_{_{b_{u_i}v}}\in \F_q }\bigg(\theta\big((k_{b_{u_i}v}-l_{b_{u_i}v})m_{b_{u_i}v}\big) \cdot\nonumber\\&&\!\!\! \!\! \prod_{(b_t,w_i)\in \J_\t}\sum_{m_{_{b_t w_i}}\in \F_q}\!\!\!\theta\big((k_{b_tw_i}-l_{b_tw_i})m_{b_tw_i}-\sum_{r=s}^m \alpha m_{b_{r}v}l_{b_{r}w_i}m_{b_tw_i}\big)\bigg)
\end{eqnarray*}
Since $L$ is a pattern matrix, we have $\sum_{r=s}^m \alpha m_{b_{r}v}l_{b_{r}w_i}m_{b_tw_i}=\alpha m_{b_{u_i}v}l_{b_{u_i}w_i}m_{b_tw_i}$ and $l_{b_tw_i}=l_{b_{u_i}v}=0$.
Then
\begin{eqnarray}
C_{_K}&=&\frac{1}{q^{2\ell}}  \prod_{i=1\atop {(b_{u_i},v)\in \J_\t}}^l\sum_{ m_{_{b_{u_i}v}}\in \F_q }\bigg(\theta(k_{b_{u_i}v}m_{b_{u_i}v}) \cdot\nonumber\\&&\!\!\! \!\!\prod_{(b_t,w_i)\in \J_\t} \sum_{m_{_{b_t w_i}}\in \F_q}\!\!\!\theta\big((k_{b_tw_i}-\alpha m_{b_{u_i}v}l_{b_{u_i}w_i})m_{b_tw_i}\big)\bigg)\label{ck3}
\end{eqnarray}
It is easy to see $K$ can have possible nonzero entries different from $L$ on positions $(b_{u_i},v)$ and $(b_t,w_i)$ for $1\leqslant i \leqslant \ell$. Notice that
those are all on the $(b_{u_i},w_i)$-hook with   $(b_{u_i}, w_i)\in Y\subset\p$.
Then by Corollary \ref{calcor}, using truncated column operations and truncated column operations, we  know: For each $e_{_K}$ such that $C_K\not=0,$   there exists $u_{_K}\in U^w\cap U$ such that $e_{_K}=e_{_L}\circ u_{_K}\in M_\O.$ Therefore we obtain $e_{_L}\circ g\in M_\O.$ Moreover, by Lemma \ref{patterninvariant} and Proposition \ref{irr}, we obtain $M_\O$ is an irreducible $FU$-module.
\end{proof}

\begin{Remark}
More precisely, we can actually determine the coefficient $C_K$. Fix $m_{b_{u_i}v}$ in (\ref{ck3}), then by (\ref{orth}), $$\sum_{m_{_{b_t w_i}}\in \F_q}\theta\big((k_{b_tw_i}-\alpha m_{b_{u_i}v}l_{b_{u_i}w_i})m_{b_tw_i}\big)\not=0$$ implies $k_{b_tw_i}-\alpha m_{b_{u_i}v}l_{b_{u_i}w_i}=0$, that is $m_{b_{u_i}v}=\alpha^{-1} l_{b_{u_i}w_i}^{-1} k_{b_tw_i}$. Therefore
$$C_K=\frac{1}{q^{|Y|}}\prod_{(b_{u_i},w_i)\in Y} \theta(\alpha^{-1} l_{b_{u_i}w_i}^{-1}k_{b_{u_i}v} k_{b_tw_i}).$$
\end{Remark}

\begin{Remark}For  $\lambda=(n-m,m)\vdash n:$ By general theory every batch $\M_\t$ of $M^\lambda$ contains precisely one trivial component and  this is the orbit module with empty pattern. More precisely, the unique pattern matrix $L$ in $\X_\t,$ whose only nonzero entries are the last ones, induces the unique trivial component of the  $\t$-batch $\M_\t$. This is given as $M_{\t}=Fe_{_L}$.
\end{Remark}

\section{The Specht modules $S^{(n-m,m)}$ }
Let $\lambda=(n-m,m)\vdash n$. Having completely decomposed  $M^\lambda$ into a direct sum of irreducible $FU$-modules, we now turn
our attention to the unipotent Specht module $S^\lambda$ given by  James's kernel intersection theorem.

\subsection{The homomorphism $\Phi_m$}
\begin{Defn}\label{phi1i} Assume that $0\leqslant i\leqslant m.$ Define $\phi_{_{1,i}}$ to be the linear map from $M^{(n-m,m)}$ into $M^{(n-i,i)}$
which sends each $m$-dimensional subspace $V$ to the formal linear combination of the $i$-dimensional subspaces contained in it. More precisely, let $X\subseteq V=\F_q^n$ with $\dim X=m.$ Then $$\phi_{_{1,i}}([X])=\sum\limits_{\mbox{\tiny $\begin{array}{c} Y\subseteq X\\ \dim Y =i\end{array}$}}[Y]$$ where $[X]$ denotes the flag $X\subseteq V$ in $\mathcal F(\lambda).$\\
\end{Defn}

\begin{Theorem}[James, \cite{gj1}]\label{kerinter}
Let $\lambda=(n-m,m)$ be a 2-part partition, then:
 $$S^{\lambda}=\bigcap\limits_{i=0}^{m-1}\ker\phi_{_{1,\,i}}, \quad
 \dim\, S^{\lambda}= \begin{bmatrix} n\\m\end{bmatrix}\,-\begin{bmatrix} n\\m-1\end{bmatrix}.$$
\end{Theorem}

We now concentrate on one of these homomorphisms,
$\phi_{_{1,m-1}}: M^{\lambda} \rightarrow M^{\mu}$
where $ \lambda=(n-m,m), \mu=(n-m+1,m-1).$
In section \ref{permutation} we have seen that each subspace $X$ of $V$ of dimension $m$ may be given uniquely by
a row reduced $(m\times n)$-matrix $M\in\Xi_{m,n}$. Hence we can translate the definition for $\Phi_m:=\phi_{_{1,m-1}}$ in \ref{phi1i} into the language of matrices:

\begin{Prop} \label{translate}
Let $\lambda=(n-m,m), \mu=(n-m+1,m-1),$ and let $M\in\Xi_{m,n},\,\tab(M)=\t$ with $ \underline \t=(b_1,\dots,b_m).$ For any $1\leqslant d \leqslant m$ we define
$R_d(M)$ to be the set of all $(m-1)\times n$-matrices obtained from  adding multiplies of row $b_d$ to rows $b_t$ for $t=d+1,\ldots,m$ and then deleting
row $b_d$ from $M$. Then $\Phi_m([M])$ is given as formal linear combination: $$\Phi_m([M])=\sum_{d=1}^m \sum\limits_{N\in R_d(M)}[N]\in M^\mu. $$
Moreover $\tab(N)\in \RStd(\mu)$ for $N\in R_d(M)$ is obtained from $\t$ by  moving $b_d$
to the first row of $\t$ at the appropriate place to make the resulting $\mu$-tableau row standard, denoted by $\ud.$
\end{Prop}
\begin{proof}
This is just a  linear algebra question, so we leave it to the reader.
\end{proof}
\begin{Example}
  Let $\lambda=(2,2), \begin{pmatrix}1 &0 &0 &0\\0 &m &n &1\end{pmatrix}\in \Xi_{2,4}$. Then
$$\Phi_2:\,\begin{bmatrix}\begin{pmatrix}1 &0 &0 &0\\0 &m &n &1\end{pmatrix}\end{bmatrix}\longmapsto
\begin{bmatrix}\begin{pmatrix}1 &0 &0 &0\end{pmatrix}\end{bmatrix}+ \sum\limits_{a\in \F_q} {\begin{bmatrix}\begin{pmatrix}a &m &n &1\end{pmatrix}\end{bmatrix}}.$$
\end{Example}

\begin{Remark}\label{jsplit}Keep the notations in Proposition \ref{translate}.
If we set $\Phi_m^d([M])  =  \sum _{N\in R_d(M)}[N],$ then $\Phi_m([M])=\bigoplus_{d=1}^m \Phi_m^d([M]).$ In fact, $\Phi_m^d $ is $\Phi_m $ composed with the projection from $M^\mu$ onto the $\ud$-batch of $M^\mu$, and   hence is $FU$-linear. Now we use a picture to show the element $N\in R_d(M)$ for a fixed $d:$\\
\begin{center}
\begin{picture}(300,150)
\put(10,75){$b_d$}
\qbezier(40,0)(20,75)(40,150)
\multiput(35,75)(5,0){51}{\line(1,0){2}}
\put(45,90){\framebox(130,60){\shortstack{$m_{_{b_i j}}$\\\\$ (1 \leqslant i\leqslant d-1)$}}}
\put(90,76){omitted}
\put(45,0){\framebox(130,60){\shortstack{$m_{_{b_i j}}+\alpha_i m_{_{b_d j}}$\\\\$(d+1\leqslant i \leqslant m, j<b_d)$}}}
\put(190,76){0}
\put(190,160){$b_d$}
\put(190,142){$0$}
\qbezier[28](192,88)(192,114)(192,139)
\put(186,56){$\alpha_{_{d+1}}$}
\put(190,0){$\alpha_{_{m}}$}
\qbezier[25](192,50)(192,35)(192,10)
\put(210,0){\framebox(70,150){\shortstack{$m_{_{b_i j}}$\\\\$(b_d<j\leqslant n)$}}}
\qbezier(285,0)(295,75)(285,150)
\put(105,-25){\textbf{Picture of } $N\in R_d(M)$}
\end{picture}
\end{center}

\quad\\

Obviously $\uu=\{ {b_1},\ldots,b_{d-1}, b_{d+1}, \ldots, b_m\}$.
By Definition \ref{jt}, we have
$\J_{\t}\cap \J_{\ud}=\{(i,j)\,|\,i>j,\, i\in \uu, \, j\notin \tt\,\}.$
In particular,   $\J_{\t}\cap \J_{\ud}$
together with row $b_d$ gives $\J_\t$ and together with column $b_d$ gives $\J_{\ud}.$ That is,
$\J_{\t}=(\J_{\t}\cap \J_{\ud})\, \dot{\cup}\, \{(b_d,j)\,|\, j<b_d,   j\notin \tt\,\}$ and $\J_{\ud}=(\J_{\t}\cap \J_{\ud})\, \dot{\cup}\, \{(b_i,b_d)\,|\,i=d+1,\ldots, m \}.$
\end{Remark}

Next we shall show first $\Phi_m^d$ preserves (filled) patterns. Then  it follows that $\Phi_m$ preserves (filled) patterns since  $\Phi_m =\bigoplus_{d=1}^m \Phi_m^d .$
To begin with, we   investigate  $\Phi_m^d(e_{_L})$ for $e_{_L}\in \M_\t\subset M^\lambda.$ By Definition \ref{defofide},
we have
\begin{equation}\label{phimdel}\Phi_m^d(e_{\!_L})=\frac{1}{q^{|\mathfrak{J}_{\t}|}}
\sum\limits_{M\in\X_{\t}}\chi_{_{L}}(-M)\Phi_m^d([M]).\end{equation}
By \ref{jsplit} we may write $\Phi_m^d([M])=\sum_{N \in R_d(M)}[N]$  where
$\tab(N)=\u_{\mbox{\tiny $d$}}.$
And for $N=(n_{_{b_ij}})\in R_d(M)\subset\X_{\u_d}$, we have:
\begin{equation}\label{nbij}n_{_{b_ij}}=
\begin{cases}m_{_{b_ij}}&  \text{ if } i\leqslant d-1 \text{ or }j>b_d;\\
\alpha_i\in \F_q & \text{ if }  d+1\leqslant i \leqslant m, j = b_d;\\
m_{_{b_{i}j}}+\alpha_i m_{_{b_dj}}& \text{ if } d+1\leqslant i \leqslant m, j <b_d.\end{cases}\end{equation}
Obviously different elements in $R_d(M)$ are distinguished by the   entries $\alpha_i$ on places $(b_i,b_d)$ for $ d+1\leqslant i \leqslant m$.
Let $\underline \alpha=(\alpha_{d+1},\ldots,\alpha_m)$ and denote
$N=N_{\underline\alpha}\in R_d(M)$ with $\underline\alpha \in \F_q^{m-d}$ fixed. Then
\begin{equation}\Phi_m^d([M])=\sum\nolimits_{\underline\alpha\in \mbox{\tiny $\F_q$}^{m-d}}[N_{\underline\alpha}].\label{Nalpha}\end{equation}
From (\ref{phimdel}) we obtain
\mbox{$\Phi_m^d(e_{\!_L})
=\frac{1}{q^{|\mathfrak{J}_{\t}|}}
\sum_{\mbox{\tiny $\underline\alpha\!\!\in\!\F_q^{m-d}$}}\sum_{\mbox{\tiny $M\!\!\in\!\!\X_\t$}}\chi_{_{L}}(-M)[N_{\underline\alpha}]. $}
Now we fix an $\underline\alpha\in \F_q^{m-d}.$ Rewrite
$[N_{\underline\alpha}]=\sum_{K\in \X_{\u_{\mbox{\tiny $d$}}}} \chi_{_K}(N_{\underline\alpha})e_{_K},$
then
\begin{eqnarray}
\Phi_m^d(e_{\!_L})
=\frac{1}{q^{|\mathfrak{J}_{\t}|}}\sum\limits_{\underline \alpha\in \F_q^{m-d}}\sum_{K\in\X_{\ud}}\sum\limits_{M\in\X_{\t}}\chi_{_{L}}(-M)\chi_{_{K}}
(N_{\underline\alpha})e_{\!_{K}}:=\sum_{K\in\X_{\ud}}C_{\!_{K}}e_{\!_{K}}.
\label{phiphi}\end{eqnarray}
Using Remark \ref{jsplit} and (\ref{nbij}) we get:
\begin{eqnarray}&\!\!\!\!&\!\!\!\!\!\! \!\!\chi_{_{L}}(-M)\chi_{_{K}}(N_{\underline\alpha})=
\prod_{(b_i,j)\in \J_\t} \theta\,(-l_{_{b_ij}}m_{_{b_ij}}) \prod_{(b_i,j)\in \J_{\ud}} \theta\,(k_{_{b_ij}}n_{_{b_ij}})\nonumber\\
 &\!\!\!\!\!\!\!\!\!\!\!\!\!\!\!\!\!\!\! \!\!\!\!\!\!\!\!=& \!\!\!\! \!\!\!\!\!\!\!\!\!\!\!\!\!\!\!
\prod_{(b_i,j)\in \J_\t \cap \J_{\ud} }  \!\!\!\!\!\!\!\!\theta\,(-l_{_{b_ij}}m_{_{b_ij}})
\,\theta\,(k_{_{b_ij}}n_{_{b_ij}})\!\!    \prod_{\stackrel{ 1\leqslant j<b_d}{j\, \notin \tt}}
 \!\!\theta\,(-l_{_{b_dj}}m_{_{b_dj}})  \!\!\!   \!\!\prod_{d+1\leqslant i\leqslant m}   \!\! \!\!\! \theta\,(k_{_{b_ib_d}}\alpha_i)\quad\quad
\label{chilk}
\end{eqnarray}
and  for $(b_i,j)\in \J_\t\cap\J_{\ud}$ we have: $\theta\,(-l_{_{b_ij}}m_{_{b_ij}})\,\theta\,(k_{_{b_ij}}n_{_{b_ij}})=$
\begin{eqnarray}\label{tud}
\begin{cases}\theta\,\((k_{_{b_ij}}-l_{_{b_ij}})m_{_{b_ij}}\)&  \text{ if } i\leqslant d-1 \text{ or }j>b_d\\
\theta\,\((k_{_{b_ij}}-l_{_{b_ij}})m_{_{b_ij}}\)\, \theta\,(\alpha_ik_{_{b_ij}} m_{_{b_dj}})& \text{ if } d+1\leqslant i \leqslant m, j <b_d\,. \end{cases}
\end{eqnarray}
For $K\in\X_{\u_d},$ $ \underline \alpha\in \F_q^{m-d}$ fixed, let
$C_K^{\underline\alpha}=\frac{1}{q^{|\mathfrak{J}_{\t}|}} \sum_{M\in\X_{\t}}\chi_{_{L}}(-M)\chi_{_{K}}(N_{\underline\alpha}).$
Thus the coefficient $C_K  $  of $e_{_K}$  is
\begin{equation}\label{ck} C_K=\sum\nolimits_{\underline \alpha\in \F_q^{m-d}} C_K^{\underline\alpha}.\end{equation}
By (\ref{chilk}) and (\ref{tud}), we get: $C_K^{\underline\alpha}=$
\begin{eqnarray*}
\frac{1}{q^{|\mathfrak{J}_{\t}|}}\!\! \sum\limits_{M\in\X_{\t}}\prod_{(b_i,j)\in \J_\t\cap \J_{\u_d}}\!\!\!\!\!\!\!\!\!\theta\((k_{_{b_ij}}\!-l_{_{b_ij}})m_{_{b_ij}}\)
\!\!\!\! \!\!\!\prod_{\stackrel{ d+1\leqslant i \leqslant m}{ 1\leqslant j<b_d,\,j\notin \tt}} \!\! \!\!\!\!\!\theta(\alpha_ik_{_{b_ij}} m_{_{b_dj}})
\theta(-l_{_{b_dj}}m_{_{b_dj}})  \theta(k_{_{b_ib_d}}\alpha_i).\end{eqnarray*}
Since $\J_{\t}=(\J_{\t}\cap \J_{\ud}) \cup \{(b_d,j)\,|\, j<b_d,   j\notin \tt\,\}$ by Remark \ref{jsplit}, we obtain:
\begin{eqnarray}\nonumber\!\!\!\! \!\!\!\!\!\!\!\!\!\!\!\!\!\!\!\!\!\!C_K^{\underline\alpha}&\!\!\!=\!\!\!&\frac{1}{q^{|\mathfrak{J}_{\t}|}}\prod_{(b_i,j)\in \J_\t\cap \J_{\u_d}}\sum_{m_{_{b_ij}}\in \F_q}\theta\,\((k_{_{b_ij}}-l_{_{b_ij}})m_{_{b_ij}}\)\nonumber
\\&&\cdot\prod_{\stackrel{ d+1\leqslant i \leqslant m}{ 1\leqslant j<b_d,\, j\notin \tt}} \sum_{m_{_{b_dj}}\in \F_q}\theta\,(\alpha_ik_{_{b_ij}} m_{_{b_dj}})
\theta\,(-l_{_{b_dj}}m_{_{b_dj}})  \theta\,(k_{_{b_ib_d}}\alpha_i).\quad\quad\label{ckn}
\end{eqnarray}
Inserting this formula into (\ref{ck}), we get:
\begin{eqnarray}\nonumber \!\!\!\!\!\!\!\!\!\!\!\! C_K
&\!\!\!\!=\!\!\!\!&\frac{1}{q^{|\mathfrak{J}_{\t}|}}\prod_{(b_i,j)\in \J_\t\cap \J_{\u_d}}\sum_{m_{_{b_ij}}\in \F_q}\theta\,\((k_{_{b_ij}}-l_{_{b_ij}})m_{_{b_ij}}\)\nonumber
\\
&\!\!& \cdot\!\!\prod_{\stackrel{ d+1\leqslant i \leqslant m}{ (b_d,j)\in \J_\t}}\sum_{m_{_{b_dj}}\in \F_q}\!\!\!\Big(\!\sum_{\alpha_i\in \F_q}
\!\theta\(\alpha_i(k_{_{b_ij}} m_{_{b_dj}}+k_{_{b_ib_d}})\)\!\Big)\theta(-l_{_{b_dj}}m_{_{b_dj}}).\quad\,\,
\label{cknew}
\end{eqnarray}
Obviously by (\ref{cknew}) $C_K$ contains the factor
\begin{equation}\label{firstfactor}\frac{1}{q^{|\mathfrak{J}_{\t}|}}\prod_{(b_i,j)\in \J_\t\cap \J_{\u_d}}\sum_{m_{_{b_ij}}\in \F_q}\theta\,\((k_{_{b_ij}}-l_{_{b_ij}})m_{_{b_ij}}\)\end{equation}
and there is no other factor of $C_K $ involving $m_{_{b_ij}}$ with $(b_i,j)\in \J_\t\cap \J_{\u_d}$. Hence
if the coefficient $C_K \not=0$, by (\ref{orth}) we must have:
\begin{equation}k_{_{b_ij}}=l_{_{b_ij}},\,\forall\,(b_i,j)\in \J_\t\cap \J_{\u_d}.  \label{commonplace}\end{equation}
That is, the entries in $K$ are the same as $L$ in the northwest, southwest and east boxes (c.f.   \ref{jsplit}). Thus the factor (\ref{firstfactor}) becomes $\frac{q^{|\mathfrak{J}_{\t}\cap \J_{u_d}|} }{q^{|\mathfrak{J}_{\t}|}}.$
In particular by (\ref{commonplace}) the remaining factor of $C_K$ is:
\begin{eqnarray}
\label{restfactor}
\prod_{\stackrel{ d+1\leqslant i \leqslant m}{(b_d,j)\in \J_\t}}\sum_{m_{_{b_dj}}\in \F_q}\Big(\sum_{\alpha_i\in \F_q}
\theta\,\(\alpha_i(l_{_{b_ij}} m_{_{b_dj}}+k_{_{b_ib_d}})\)\Big)\theta\,(-l_{_{b_dj}}m_{_{b_dj}}).\quad\end{eqnarray}


\begin{Lemma}\label{perservecor}
Let $\lambda=(n-m,m). $ Let $\O\subseteq M^{\lambda} $ associated with pattern $\p=\p(\O)$.
Then for any $e_{_K}\in \O$ we have:
$\Phi_m^d(e_{_K})=0 \text{ for any } b_d\in \ppi=\{i\,|\,(i,j)\in \p \text{ for some } 1\leqslant j \leqslant n\}.$
\end{Lemma}
\begin{proof}   Assume $(b_d,v)\in \p$. keeping   notations in  $\ref{jsplit}$
we rewrite the factor (\ref{restfactor}) as  follows
\begin{eqnarray}
&&\prod_{\stackrel{ d+1\leqslant i \leqslant m}{ 1\leqslant j<b_d \atop j\notin \,\tt, j\not= v}}
\sum_{\alpha_i\in \F_q}\sum_{m_{_{b_dj}}\in \F_q}
 \,\theta\,(\alpha_i l_{_{b_ij}} m_{_{b_dj}})
\theta\,(-l_{_{b_dj}}m_{_{b_dj}})  \theta\,(k_{_{b_ib_d}}\alpha_i)\nonumber\\
&\cdot&\prod_{ d+1\leqslant i \leqslant m} \,\sum_{\alpha_i\in \F_q}  \sum_{m_{_{b_dv}}\in \F_q}\theta\,(\alpha_il_{_{b_iv}} m_{_{b_dv}})
\theta\,(-l_{_{b_dv}}m_{_{b_dv}})  \theta\,(k_{_{b_ib_d}}\alpha_i).\quad\quad\quad\label{restfactorcase2}\end{eqnarray}
Since $L$ is a pattern matrix,
$l_{_{b_i v}}=0$ for $d+1\leqslant i \leqslant m.$ Thus (\ref{restfactorcase2}) becomes:
\begin{eqnarray*}
&&\prod_{\stackrel{ d+1\leqslant i \leqslant m}{ 1\leqslant j<b_d \atop j\notin \,\tt, j\not= v}}\,
\sum_{\alpha_i\in \F_q\atop{\,\mbox{\scriptsize$m$}_{_{\mbox{\tiny$b_dj$}}}\in \F_q}}\!\!\!\!\!\!
 \theta\,(\alpha_i l_{_{b_ij}} m_{_{b_dj}})
\theta\,(-l_{_{b_dj}}m_{_{b_dj}})  \theta\,(k_{_{b_ib_d}}\alpha_i)\!\!\!\!
 \sum_{m_{_{b_dv}}\in \F_q}\!\!\!\!\!\!
\theta\,(-l_{_{b_dv}}m_{_{b_dv}}) \end{eqnarray*}
Note that $m_{_{b_dv}}$ only occurs in the second sum,  hence  by (\ref{orth}), if $C_K\not=0$ we must have $l_{_{b_{_{d}}v}}=0$ which is a contradiction to $(b_d,v)\in \p$. It means that  $C_{_K}=0$ for all $e_{\!_{K}}\in \X_{\ud}$ and the claim follows from \ref{phiphi}.
\end{proof}

Now we are ready for the following theorem:
\begin{Theorem} \label{keepcon}Let $\lambda=(n-m,m), \mu=(n-m+1,m-1).$ Then the homomorphism $\Phi_m: M^\lambda \rightarrow M^\mu$ preserves (filled) patterns. More precisely, let $\t\in\RStd(\lambda) $ and  $e_{_L}\in \O\subset \M_{\t}\subset M^\lambda,$ $\pf=\pf(\O)$ be the filled pattern of $\O.$ Then
$\Phi_m(e_{_L})=\sum_K C_{K}e_{_K}$ where $K\in \Xi_{m-1,n}$ satisfies:
Each $e_{_K}$ with $C_{K}\not=0$ is contained in an orbit $\tilde \O$ of some batch of $M^\mu$ such that
$\pf(\tilde \O)=\pf.$
\end{Theorem}
\begin{proof} Keeping notations in \ref{jsplit}. Since $\Phi_m =\bigoplus_{d=1}^m \Phi_m^d,$  it suffices to prove  $\Phi_m^d$ preserves filled patterns. Since  $\Phi_m^d$ is $FU$-linear and each orbit module is an irreducible $U$-module by \ref{irr}, we can restrict our attention to the case that $L=(l_{_{b_ij}})\in \X_\t$ is a pattern matrix. Moreover, by \ref{perservecor}, we only need to consider those $d$ such that $b_d\notin \p$. Thus $l_{_{b_dj}}=0,$ for all $(b_d,j)\in \J_\t$.
Assume $\Phi_m^d(e_{_L})=\sum_{K} C_{K}e_{_K}$ where $K\in \Xi_{m-1,n}$ and $ C_K\not=0$.
The remaining factor (\ref{restfactor})
becomes:
\begin{equation}\label{restfactorcase1}\prod_{\stackrel{ d+1\leqslant i \leqslant m}{ (b_d,j)\in \J_\t}}\sum_{\alpha_i\in \F_q}\, \sum_{m_{_{b_dj}}\in \F_q}\theta\,\(\alpha_i(l_{_{b_ij}} m_{_{b_dj}}+k_{_{b_ib_d}})\).\end{equation}

\Case {1: $(b_i, j)\notin \p$ for all $d+1\leqslant i \leqslant m,  (b_d,j)\in \J_\t$}
In this case $l_{_{b_{_{i}}j}}=0,$ thus
(\ref{restfactorcase1}) is a nonzero multiple of $\prod_{{ d+1\leqslant i \leqslant m}}\sum_{\alpha_i\in \F_q}  \theta\,(k_{_{b_ib_d}}\alpha_i).$
Therefore by (\ref{orth}),   $C_K\not=0$ implies
$k_{_{b_i b_d}}=0$ for all $d+1\leqslant i \leqslant m.$ Combining with (\ref{commonplace}), we obtain easily in this case $K$
is a pattern matrix and $\pf(K)=\pf.$

\Case {2: There exists  $(b_u,v)\in \p$ for some $d+1\leqslant u \leqslant m, (b_d,v)\in \J_\t$}
In this case, we rewrite the factor (\ref{restfactorcase1}) of $C_K$ by separating the elements in the filled pattern from those which are not:
\begin{eqnarray*}
\prod_{ d+1\leqslant u \leqslant m \atop { (b_d,v)\in \J_\t \atop{(b_u,v)\in \p}}}\,\sum_{m_{_{b_dv}}\in \F_q}\sum_{\alpha_u\in \F_q} \!\!\theta \((l_{_{b_{_{u}}v}}m_{_{b_dv}}+k_{_{b_{_{u}}b_{_d}}})\alpha_u\)\cdot\!\!\!\!\!\!\prod_{d+1\leqslant i\leqslant m\atop { (b_d,j)\in \J_\t \atop{(b_i,j)\notin \p}}}\!\sum_{m_{_{b_dj}}\in \F_q}\sum_{\alpha_i\in \F_q} \theta (k_{_{b_{_{i}}b_{_d}}}\alpha_i).\end{eqnarray*}
Hence by (\ref{orth}), $C_K\not=0$ implies
\begin{eqnarray}\label{mbdvfixed}
 \begin{cases}
  k_{_{b_{_{u}}b_{_d}}}=-l_{_{b_{_{u}}v}}m_{_{b_dv}}&\text{ for }d+1\leqslant u \leqslant m,\,(b_d,v)\in \J_\t,(b_u,v)\in \p\\
  k_{_{b_{_{i}}b_{_d}}}=0 &\text{ for }d+1\leqslant i \leqslant m, \, (b_i,j)\notin  \p,\,\forall\,(b_d,j)\in \J_\t.\quad\quad
 \end{cases}
\end{eqnarray}
Note that
$(b_u,b_d)$  is on the $(b_u,v)$-hook arm with nonzero entry $l_{_{b_u v}}\in \pf$ in the hook corner. Hence  by Corollary \ref{calcor}
and (\ref{commonplace}), we know that $C_K\not=0$ implies that
$e_{_K}$ is  contained in an orbit $\tilde \O$ of $\tab(K)$-batch of $M^\mu$ such that
$\pf(\tilde \O)=\pf.$
\end{proof}

Now we collect   all orbit modules  $M_\O$, where $\O\subseteq \M_\t$ is some orbit associated with some fixed filled pattern $\pf$ and $\t$ runs through  the tableaux in $\RStd(\lambda)$ satisfying $\ppi\subseteq \tt$ and $\pj\cap \tt=\emptyset$:
\begin{Defn}\label{collection}Let $\lambda=(n-m,m)\vdash n$, $\p$ be a $\lambda$-pattern and let $\pf$ be a filling of $\p$. Define:
\begin{eqnarray*}
\mathfrak{C}_{\pf}^\lambda &=& \bigoplus\limits_{\pf(\O)= \pf} M_{\mathcal{O}}=
\bigoplus\limits_{\pf(\O)=\pf}\bigoplus_{\mbox{\small $e$}\!_{_{\mbox{\tiny $L$}}}\in \mathcal{O}} Fe_{_L} , \text{ where }M_{\mathcal{O}}\subset M^\lambda,
\end{eqnarray*}
runs through all the different  orbits in $M^\lambda$ which have the same filled pattern $\pf$.
\end{Defn}

Recall the short notation for an orbit module in   \ref{short for orbit}.
\begin{Example}Let $\lambda=(3,3),  \pf=\{l_{41}\not=0\}.$ Then:
\begin{eqnarray*}
\mathfrak{C}_{\pf}^\lambda&=&\begin{pmatrix}\ast &1 &0 &0 &0 &0\\\ast &0 &1 &0 &0 &0\\l_{41} &0 &0 &1 &0 &0\end{pmatrix}\oplus
\begin{pmatrix}\ast&1 &0 &0 &0 &0\\l_{41} &0 &\ast &1 &0 &0\\0&0&0&0&1&0\end{pmatrix}\\& \oplus &
\begin{pmatrix}\ast &0+\square &1 &0 &0 &0\\l_{41} &\ast &0 & 1&0 &0\\0&0&0&0&1&0\end{pmatrix}\oplus
\begin{pmatrix}\ast &1 &0 &0 &0 &0\\l_{41} &0 &\ast & 1&0 &0\\0&0&0&0&0&1\end{pmatrix}.
\end{eqnarray*}
\end{Example}
From   \ref{keepcon} we   easily obtain the following corollary :
\begin{Cor}\label{keepcon1}Let $\lambda=(n-m,m)\vdash n, \mu=(n-m+1,m-1)\vdash n$ and $\O$ be an orbit in $M^\lambda$ with the filled pattern
$\pf=\pf(\O)$. Then
$\Phi_m(\mathfrak{C}_{{\pf }}^{\lambda})\subseteq \mathfrak{C}_{{\pf }}^{\mu}.$
\end{Cor}

\begin{Prop}\label{sphim}If $\Char(F)=0$, then $S^{(n-m,m)}=\ker \Phi_m.$
\end{Prop}
\begin{proof}By   \ref{kerinter}, we have
$S^{(n-m,m)}=\big(\bigcap\nolimits_{i=0}^{m-2}\ker\phi_{_{1,\,i}}\big) \,\bigcap\,\ker \Phi_m.$
Hence it suffices to prove $\ker  \Phi_m \subset \ker\phi_{_{1,\,i}}$ for all $0\leqslant i \leqslant m-2.$
In fact, for $X\subseteq V, \,\dim_{\mbox{\tiny $\F_q$}}X=m,$ we have by  \ref{phi1i}:
$$\Phi_m([X])=\sum\limits_{\mbox{\tiny $\begin{array}{c} Y\subseteq X\\ \dim Y =m-1\end{array}$}}[Y], \quad\quad
\phi_{1,\,i}([Y])=\sum\limits_{\mbox{\tiny $\begin{array}{c} Z\subseteq Y\\ \dim Z =i\end{array}$}}[Z].\quad \text{ Then} $$
 $$   \phi_{1,\,i}\circ \Phi_m([X]) = \sum\limits_{\mbox{\tiny $\begin{array}{c} Y\subseteq X\\ \dim Y =m-1\end{array}$}}
\sum\limits_{\mbox{\tiny $\begin{array}{c} Z\subseteq Y\\ \dim Z =i\end{array}$}}[Z].$$
Now we calculate the number of the $(m-1)$-dimensional subspace \mbox{$Y\subseteq X$}  which contains a fixed $i$-dimensional space $Z\subseteq X$. Obviously, this number equals the ways of choosing $(m-i-1)$-dimensional spaces from an \mbox{$(m-i)$-}dimensional space,
that is $[\begin{smallmatrix}m-i\\m-i-1\end{smallmatrix}]=[m-i].$ Hence for all $0\leqslant i \leqslant m-2,$ we have:
$\,\phi_{1,\,i}\circ \Phi_m = [m-i]\,\phi_{1,\,i}.$
Since $\Char(F)=0$, we obtain
for all $0\leqslant i \leqslant m-2:$
$ \phi_{1,\,i}=\frac{1}{[m-i]}\,\phi_{1,\,i}\circ \Phi_m $ and hence  $\ker\Phi_m \subset  \ker\phi_{1,\,i}$ for all $0\leqslant i \leqslant m-2.$ Therefore $S^{(n-m,m)}=\ker \Phi_m.$
\end{proof}

\begin{Cor}\label{epi} If $\Char(F)=0$ then $\Phi_m$ is an epimorphism.
\end{Cor}
\begin{proof}By   \ref{sphim}, if $\Char(F)=0$, then
$\dim\Phi_m \big(M^{(n-m,m)}\big)=\dim M^{(n-m,m)}-\dim \ker \Phi_m
=\dim\,M^{(n-m,m)}-\dim\, S^{(n-m,m)}\stackrel{(\ref{kerinter})}{=}
[\begin{smallmatrix} n\\m\end{smallmatrix}]-\,([\begin{smallmatrix} n\\m\end{smallmatrix}]\,-[\begin{smallmatrix} n\\m-1\end{smallmatrix}])=[\begin{smallmatrix} n\\m-1\end{smallmatrix}]=\dim\, M^{(n-m+1, m-1)}.$
Obviously, $\Phi_m \big(M^{(n-m,m)}\big)\subseteq M^{(n-m+1, m-1)},$ hence $\Phi_m \big(M^{(n-m,m)}\big)=M^{(n-m+1, m-1)}.$ 
\end{proof}

There is an easy consequence of   \ref{keepcon1} and      \ref{epi}:
\begin{Cor}\label{f0keepcon1}Let $\lambda=(n-m,m)\vdash n, \mu=(n-m+1,m-1)\vdash n$ and $\O$ be an orbit in $M^\lambda$ with filled pattern
$\pf=\pf(\O)$. If $\Char(F)=0$ then
$\Phi_m(\mathfrak{C}_{{\pf }}^{\lambda})=\mathfrak{C}_{{\pf }}^{\mu}.$
\end{Cor}

\begin{Defn}Let \mbox{$\lambda=(n-m,m)\vdash n, \,\mu=(n-m+1,m-1)\vdash n$.} And let $\pf$ be a filled pattern.
Define: $\Phi_{m,\pf}: \mathfrak{C}_{{\pf }}^{\lambda}\rightarrow \mathfrak{C}_{{\pf }}^{\mu}$.
Observe that \mbox{$\Phi_m=\bigoplus_{\pf} \Phi_{m,\pf}.$}
\end{Defn}

\begin{Theorem}\label{iso}Let $\lambda, \mu$ be 2-part partitions of $n.$ Let $\O $ (resp. $ O'$) be an orbit in $M^\lambda$ (resp. $M^ \mu$).
If the filled patterns of this two orbits are the same, then the corresponding irreducible orbit modules $M_{\O}$ and $ M_{\O'}$ are isomorphic.
\end{Theorem}
\begin{proof}Let $m=[\frac{n}{2}].$ With respect of the dominance order $\unrhd$ of partitions, we have
$(n-m,m)\unrhd (n-m+1,m-1)\unrhd \cdots \unrhd (n,0)$.
Let $\p$ (resp. $\pf$) be a (filled) pattern of some orbit in $M^{(n-m,m)}$ and let $s=|\p|.$ It is obvious that
this filled pattern only fits the following partitions:
$(n-m,m)\unrhd \mbox{(n-m+1,m-1)}\unrhd \cdots \unrhd (n-s,s).$
Moreover  $\pf$ only fits one $\t$-batch $\M_\t$ of $M^{(n-s,s)}$ since the elements in the set
$\ppi=\{i\,|\,(i,j)\in \p \text{ for some }1\leqslant j \leqslant n\}$ should be in the second row $\tt$ of $\t$ but $|\ppi|=s$ hence these are all elements in $\tt$, which leads to $\t$ is fixed. Thus, we obtain
$\CC_{\pf}^{(n-s,s)}=M_\O$ where $\O$ is the unique orbit in $\M_\t$ such that $\pf(\O)=\pf.$ We prove the claim by induction.

Suppose $O'$ is an arbitrary orbit in $M^{(n-s-1, s+1)}$ such that $\pf(\O')=\pf.$ Then by \ref{keepcon1}, we get
$\Phi_{s+1} (M_{ \O'})\subseteq \CC_{\pf}^{(n-s,s)}=M_\O$ where $M_{\O'}$ is the orbit module corresponding to $\O'.$
By Theorem \ref{uinvariant}, $M_{\O}$ is an irreducible $FU_n$-module and obviously $\Phi_m (M_{ \O'})\not=0$ hence
we obtain  $\Phi_{s+1} (M_{ \O'})=M_{  \O}.$
Since $M_{\O'}$  is also  irreducible, 
we get $M_{\O'}\cong M_{\O}.$

Assume  for some $i \geqslant s+1,$ $M_{\tilde \O}\cong M_\O$ for all \mbox{$\tilde \O\subset M^{(n-i, i)}$} such that $\pf(\tilde \O)=\pf$.
Suppose $\O''$ is an arbitrary orbit in $M^{(n-i-1, i+1)}$ such that $\pf( \O'')=\pf,$
Again by \ref{keepcon1}, we get
$\Phi_{i+1} (M_{ \O''})\subseteq \CC_{\pf}^{(n-i,\,i)}= \bigoplus M_{\tilde \O}$ where $M_{\tilde \O}\subseteq M^{(n-i, i)} $ and $\pf(\tilde \O)=\pf.$
Since $M_{ \O''}$ is an irreducible $FU_n$-module by   \ref{uinvariant} and $\Phi_{i+1} (M_{ \O''})\not=0$, we obtain
$M_{ \O''}\cong \Phi_{i+1} (M_{ \O''})\subseteq  \bigoplus M_{\tilde \O}\cong\bigoplus M_{\O}.$
Since $M_{\O}$ is irreducible, we have $M_{ \O''}\cong M_{ \O}.$
\end{proof}
\subsection{Special orbits in $M^{(n-m,m)}$}\label{specialorbit}
In this section we investigate two special orbits in $M^{(n-m,m)}$
which can  easily give us some elements in $S^{(n-m,m)}$.
\begin{Prop}\label{fullcon}Let $\lambda=(n-m,m).$ Let $\O$ be an orbit in the $\t$-batch $\M_\t$ of $ M^\lambda$.
Let $\p=\p(\O)$ be the pattern of $\O$. If $|\p|=m,$ then
$M_\O\subset \ker \Phi_m$  and $\t\in\Std(\lambda).$
More precisely, for any $e_{_L}\in \O$ we have $\Phi_m(e_{_L})=0 $ and $\tab(L)\in\Std(\lambda).$
\end{Prop}
\begin{proof}
Recall that $\ppi =\{i\,|\,(i,j)\in \p \,\},\, \pj =\{j\,|\,(i,j)\in   \p \,\}$. Thus $|\p|=m$ says that each row of $L$ has a nonzero entry besides the last 1's. Hence by \ref{perservecor} we have $\Phi_m^d(e_{_L})=0$ for $d=1,\dots,m$ and hence $\Phi_m(e_{_L})=0$ for all $e_{_L}\in \O$ by
 \ref{jsplit}. Thus, $M_\O\subset \ker \Phi_m$.
Now let $e_{_L}\in \O$ be the  pattern idempotent in $\O$. Assume $$\tab(L)=\begin{tabular}{|c|c|c|c|c|c|}\hline $a_1$ & $a_2$&$\cdots$&$a_m$&$\cdots$&$a_{n-m}$\\
\hline $b_1$ & $b_2$&$\cdots$&$b_m$\\\cline{1-4} \end{tabular}\,\,.$$
Since $|\p(L)|=|\p|=m,$ we need to have $i$-many columns $a_1,\ldots,a_i$ before column $b_i$, hence $a_i<b_i$
and $\t=\tab(L)$ is a standard $\lambda$-tableau.
\end{proof}

From the definition of the homomorphisms $\phi_{1,\,i}$ for \mbox{$0\leqslant i \leqslant m-2,$} we know
the orbit modules with full pattern also live in \mbox {$\ker \phi_{1,\,i},\,\forall\,\leqslant i \leqslant m-2,$} then
they are in the Specht module $S^{(n-m,m)}$  for any arbitrary field.

Note that the result $\tab(L)\in \Std(\lambda)$ in Proposition \ref{fullcon}  coincides with an important
result by Sin\'{e}ad Lyle, which we will use very often in the later sections. First we introduce an order
which was used in Lyle's theorem:

\begin{Defn}Let $\lambda=(n-m,m).$ Define a partial order $\trianglelefteq$ on $\RStd(\lambda)$\vspace{1.2mm}
by:
\centerline{ $\begin{tabular}{|c|c|c|c|c|c|}\hline $a_1$ & $a_2$&$\!\!\cdots\!\!$&$a_m$&$\!\!\cdots\!\!$&$a_{n-m}$\\
\hline $b_1$ & $b_2$&$\!\!\cdots\!\!$&$b_m$\\\cline{1-4} \end{tabular}\trianglelefteq
\begin{tabular}{|c|c|c|c|c|c|}\hline $a'_1$ & $a'_2$&$\!\!\cdots\!\!$&$a'_m$&$\!\!\cdots\!\!$&$a'_{n-m}$\\
\hline $b'_1$ & $b'_2$&$\!\!\cdots\!\!$&$b'_m$\\\cline{1-4} \end{tabular}\,\Leftrightarrow b_i\leqslant b'_i, \,\forall\,1\leqslant i \leqslant m.$}
\end{Defn}

\begin{Theorem}[Lyle, \cite{ly}]\label{ori}Suppose that $0\neq v\in S^{(n-m,m)}$ and write
\mbox{$v =\sum\nolimits_{X\in\Xi_{m,n}}C_XX$} where $C_X\in F.$ Say that $X$ occurs in $v$ if $C_X\neq0$.
Assume $X'$ occurs in $v$  such that for every $X$ with $X\neq X'$ and $\tab(X')\trianglelefteq \tab(X)$
we have: $X$ does not occur in $v$. Then $\tab(X')$ is standard.
\end{Theorem}
Recall the order we defined in section 2, (c.f. \ref{original basis}). Since our order is weaker than the order in Lyle's theorem, we can obtain the following corollary, on which our work heavily relies:
\begin{Cor}\label{lyle}Suppose that $0\neq v\in S^{(n-m,m)}$. Then $\last(v)$ is standard.
\end{Cor}

Now we investigate another special orbits having empty pattern.
First we prove an easy lemma which will be very useful later on.
\begin{Lemma}\label{fst}
For $\lambda=(n-m,m)\vdash n$, $ \mu=(n-m+1,m-1)\vdash n$, let
$P_{m}=\RStd(\lambda)\setminus \Std(\lambda)$, $ Q_{m}=\RStd(\mu).$
Then $|P_{m}|= |Q_{m}|.$
\end{Lemma}
\begin{proof}Note that $|\RStd(\lambda)|=(\begin{smallmatrix}n\\m\end{smallmatrix}),\, |\RStd(\mu)|=(\begin{smallmatrix}n\\m-1\end{smallmatrix})$
 and $|\Std(\lambda)|=(\begin{smallmatrix}n\\m\end{smallmatrix})-(\begin{smallmatrix}n\\m-1\end{smallmatrix}),$
 hence the statement holds.
\end{proof}

Recall that for $v =\sum\nolimits_{X\in\Xi_{m,n}}C_XX$, $\ttop(v)$ is the collection of all $X$ occurring in this sum with $\tab(X)=\last(v)$,  (c.f. \ref{original basis}).

Note that the trivial $FU$-module occurs in each batch of $M^\lambda$ precisely once as composition factor. This follows immediately from the Mackey decomposition, c.f. \ref{xj}. Indeed this trivial $FU$-component is the unique orbit module $M_\O$ such that $\p(\O)=\emptyset.$
\begin{Prop}\label{trivialorbit}
Let $\lambda=(n-m,m)\vdash n,\,\t\in \Std(\lambda).$ Suppose $M_\emptyset = Fe_{_L}$ is the unique trivial orbit in
the $\t$-batch $\M_\t.$ If $\Char F=0$ then there exist $v\in S^\lambda$ such that $\ttop(v)=e_{_L}.$
\end{Prop}
\begin{proof}For each $\s$-batch we denote the basis element in the empty orbit
by $L^\s_\emptyset$. We claim the set
$R:=\{\Phi_m(e_{_{L^{\s}_\emptyset}})\,|\,\s\in \RStd(\lambda)\setminus \Std(\lambda)\}$ is linearly independent.
In fact if we have a linear combination
$\sum_\s a_\s\Phi_m(e_{_{L^\s_\emptyset}})=0$ then $\Phi_m(\sum_\s a_\s e_{_{L^\s_\emptyset}})=0$ hence
$\sum_\s a_\s e_{_{L^\s_\emptyset}}\in \ker \Phi_m.$ If $\Char F=0$ then by \ref{sphim} $\ker \Phi_m=S^\lambda.$ Hence
$\sum_\s a_\s e_{_{L^\s_\emptyset}}\in S^\lambda\text{ but }\tab(L^\s_\emptyset)=\s \text{ is nonstandard  for all } L^\s_\emptyset.$
By Corollary \ref{lyle} we obtain $a_\s=0$ for all $\s\in \RStd(\lambda)\setminus \Std(\lambda).$

Let $\mu=(n-m+1,m-1)$, then   $\Phi_m(\CC ^\lambda _\emptyset)\subset \CC ^\mu_\emptyset$ and we get $R\subset \CC ^\mu_\emptyset.$
Moreover by  \ref{fst}, we know $\dim R = P_m=Q_m=\dim \CC ^\mu_\emptyset.$ Thus we obtain
$FR= \CC ^\mu_\emptyset.$ Now suppose $\t\in \Std(\lambda)$, then by \ref{keepcon1}
$\Phi_m(e_{_{L^\t_\emptyset}})\subset \CC ^\mu_\emptyset=FR $ thus
$\Phi_m(e_{_{L^\t_\emptyset}})=\sum_\s a_\s\Phi_m(e_{_{L^\s_\emptyset}})$ where $\s\in \RStd(\lambda)\setminus \Std(\lambda),\, a_\s\in F.$
Let $v=e_{_{L^\t_\emptyset}}-\sum_\s a_\s e_{_{L^\s_\emptyset}}.$ Then $v\in \ker \Phi_m=S^\lambda$, since $\Char F=0$. Moreover, we have $\last(v)=\t$ since
by \ref{lyle} we know $\last(v)$ must be standard. That is,  $\ttop(v)=e_{_{L^\t_\emptyset}}.$
\end{proof}

\subsection{Standard basis of $S^{(n-m,m)}$}
Throughout this section, we fix $\lambda=(n-m,m)\vdash n$. We shall  first construct a basis for  $S^{\lambda}$ over a field with characteristic zero and
then show  this is an integral basis for any arbitrary field. The idea is  reducing nonempty    pattern case to the second special case in the previous section, that is when the pattern is empty. In this sense, we define the following map $\R_{\p}$ where $\p$ is a  pattern. This map removes every row and column related to the   pattern $\p.$ More precisely:

\begin{Defn}\label{defofpsi} let $\p$ be a $\lambda$-pattern and $\pf$ be a filling of $\p$. Let $s=|\p|$.
For $L\in \Xi_{m,n}$, let $\t=\tab(L)$. If $\ppi\subseteq \tt$ and $\pj\cap \tt=\emptyset$, then we define $\mathfrak R_{\p}(L)$ by deleting from $L$ all rows and columns $b_i\in \ppi$ and in addition all columns $j\in \pj$. Otherwise we define $\mathfrak R_{\p}(L)=0$.  Obviously, $\mathfrak R_{\p}(L)\in  \Xi_{m-s,n-2s}$. Now we  extend this by  linearity to an $F$-linear map: $M^\lambda\longrightarrow M^\nu$ where $\nu=(n-m-s,m-s)\vdash n-2s$.
\end{Defn}

 Note that for  a pattern matrix $L$ to pattern $\p$, we have    $ \mathfrak R_{\p}(L)$ is the pattern matrix in $\Xi_{m-|\p|,n-2|\p|}$ with empty  pattern.

\begin{Example}\label{4,4}
Let $\p=\{(5, 2), (8,6)\}$ be a pattern and suppose
$$\begin{matrix}

  L=&\begin{pmatrix}          0 & 0 & 1 & 0 & 0 & 0 & 0 & 0\\
                         0 & l_{52} & 0 & 0 & 1 & 0 & 0 & 0\\
                         0 & 0 & 0 & 0 & 0 & 0& 1 & 0\\
                         0 & 0 & 0 & 0 & 0 & l_{86}& 0 & 1
\end{pmatrix}&\begin{matrix} \emph{3}\\\emph 5\\\emph 7\\\emph 8\end{matrix}&
\end{matrix}$$
then
$\tilde L =\mathfrak R_{\p}(L)=
\big(\begin{smallmatrix}
0  & 1 & 0   & 0 \\0  & 0 & 0  & 1\end{smallmatrix}\big)\in \Xi_{2,4}$. Obviously, $\p(\tilde L)=\emptyset.$
\end{Example}

\begin{Remark} \label{shifted}
let $\p$ be a $\lambda$-pattern, $s=|\p|$. Assume $L\in \Xi_{m,n}$ with $0\not=\tilde L=\R_{\p}(L)\in  \Xi_{m-s,n-2s}$.
We can easily obtain   $\tab(\tilde L)$ from $\t=\tab(L)$ in the following way:
First we delete the numbers  $i\in \ppi \cup \pj$ in $\t$ and omit the resulting gaps to obtain a row standard $\nu$-tableau
 $\tilde \t $ of shape
$\nu=(n-m-s,m-s)$ filled by numbers  $\{1,2,...,n\}\setminus(\ppi\cup \pj),$ denoted by
$\tilde \t= \t\setminus (\ppi\cup \pj), $
called \textbf{shifted $\mu$-tableau}.
Assume
$\{1,2,...,n\}\setminus(\ppi\cup \pj)$ $=\{a_1,a_2,...,a_{n-2s}\}$
with  order $a_1<a_2<\cdots<a_{n-2s}.$
Replacing the numbers $a_i$ in $\tilde \t$ by $i$ instead, we get
a $\mu$-tableau $\s$ filled by numbers $1,2,\ldots,n-2s$ with $\s=\tab(\tilde L)$.
Obviously $\s$ and $ \tilde \t$ are 1-1 correspondence if we fixed the pattern $\p$. We say $\s$ and $\tilde\t$
are \textbf{$\p$-similar}, denoted by $\s \stackrel{\p}{\sim} \tilde \t.$ Of course, $\s$ is standard if and only if
$\tilde \t$ is standard.
\end{Remark}

\begin{Example}
In \ref{4,4},
$\widetilde {\tab(L)}=\begin{tabular}{|c|c|}\hline
1 & 4 \\\hline
3 & 7 \\\hline
\end{tabular}\,\stackrel{\p}{\sim} \begin{tabular}{|c|c|}\hline
1 & 3 \\\hline
2 & 4 \\\hline
\end{tabular}\,=\tab(\tilde L).$
\end{Example}

\begin{Defn} \label{tplambda}Let $\p$ be a  $\lambda$-pattern and let $\nu=(n-m-|\p|,m-|\p|).$ Denote
$T^\lambda_{\p}$  be the set of row-standard but non-standard shifted $\nu$-tableaux, which are
filled by numbers in  $\{1,2,...,n\}\setminus(\ppi\cup \pj)$. In particular, if $\p=\emptyset, $
then $T_{\emptyset}^\lambda $ is the set of row-standard but non-standard tableaux of shape $\lambda$.
\end{Defn}

\begin{Example}  \label{3,3} Let $\lambda=(3,3), \p=\{(6, 4)\},$ hence $\ppi\cup \pj=\{4,6\}$ and
$$ T^\lambda_{\p}=\bigg \{\,
\begin{tabular}{|c|c|}\hline 3 & 5\\\hline 1 & 2\\\hline \end{tabular}\,,
\begin{tabular}{|c|c|}\hline 2 & 5\\\hline 1 & 3\\\hline \end{tabular}\,,
\begin{tabular}{|c|c|}\hline 2 & 3\\\hline 1 & 5\\\hline \end{tabular}\,,
\begin{tabular}{|c|c|}\hline 1 & 5\\\hline 2 & 3\\\hline \end{tabular}\,
\bigg \}.$$
\end{Example}

\begin{Cor}\label{tssize}Let   $\p$ be a $\lambda$-pattern and $s=|\p|$. Then
$$T^\lambda_{\p}=|\RStd(\mu)| \text{ where } \mu=(n-m-s+1,m-s-1).$$
\end{Cor}
\begin{proof}It is a easy consequence of Lemma \ref{fst}.
\end{proof}

In \ref{fullcon}, we have discussed the case that $|\p|=m$, thus we  only need to investigate the following key lemma   under the condition: $0\leqslant |\p| \leqslant m-1$:
\begin{Lemma}\label{compatible}Let $\pf$ be a filled  pattern associated with a $\lambda$-pattern $\p$. Let $s=|\p|$ such that $0\leqslant s \leqslant m-1$.
If $\{\Phi_m(e_{\!_{L}})\,|\,e_{_{L}}\in \CC_{\pf}^\lambda\} $
is linearly dependent, then
$\{\Phi_{m-s}(e_{_{\R_{_\p}(L)}})\,|\,e_{_{L}}\in \CC_{\pf}^\lambda\}$
 is  linearly dependent.
\end{Lemma}
\begin{proof}
Suppose   $\sum_{r=1}^k \gamma_r\, \Phi_m(e_{_{L_r}})=0$ with $e_{_{L_1}},\ldots, e_{_{L_k}}$  being pairwise different idempotents in $\CC_{\pf}^\lambda$ and $\gamma_r\not=0$ for $r=1,\ldots,k.$ In order to keep notation simple, we denote  for $r\in \{1,\ldots,k\}$:
\begin{equation}\label{tildeel}
e_{r}=e_{_{L_r}},\quad\tilde L_r=\R_{\p} (L_r),\quad \tilde e_{r}=e_{_{\tilde L_r}},
\quad \tab(L_r)=\t_r, \quad \tilde \t_r=\t_r \setminus \ppi\cup \pj.\end{equation}
thus  $\tilde \t_r$ is a shifted tableau filled by numbers in  $\{1,\ldots,n\}\setminus \ppi\cup \pj.$

Now we fix some  $r\in \{1,\ldots,k\}$ to investigate  $\Phi_m(e_{_{L_r}})$. Hence at this moment we drop  the index $r$, which means we let $L=L_r,e=e_r,  \t=\t_r, \tilde L=\tilde L_r, \tilde e=\tilde e_r, \tilde \t=\tilde {\t_r}$. By \ref{defofide}
 we have:
\[
e=\frac{1}{q^{|\mathfrak{J}_{\t}|}}\sum\limits_{M\in\X_{\t}}\chi_{_{L}}(-M)[M]
=\frac{1}{q^{|\mathfrak{J}_{\t}|}}\sum\limits_{M\in\X_{\t}}\prod\limits_{(b_i,j)
\in\J_{\t}}
\theta(-l_{_{b_ij}}m_{_{b_i}j})[M]
\]
\mbox{where $L=(l_{_{b_ij}})\in\X_{\t}, M=(m_{_{b_ij}})\in\X_{\t}$. Suppose $\tt=(b_1,\dots,b_m)$. For $b_d\notin \ppi,$ $(1\leqslant d \leqslant m)$:}
\begin{equation}\label{phielprop}\Phi_m^d(e)=\frac{1}{q^{|\mathfrak{J}_{\t}|}}
\sum\limits_{M\in\X_{\t}}\prod\limits_{(b_i,j)\in\J_{\t}}
\theta(-l_{_{b_ij}}m_{_{b_ij}})\Phi_m^d([M]).\end{equation}

Using similar notation as in (\ref{Nalpha}) we may write
\begin{equation}\label{phimr}\Phi_m^d([M])=\sum\limits_{\underline\alpha\in  \F_q^{m-d}}[N^d_{\underline\alpha}(M)]\end{equation}
where $\underline \alpha=(\alpha_{_{d+1}},\ldots,\alpha_{_m})\in  \F_q^{m-d}.$
If we denote $N^d_{\underline\alpha}(M)=(n^{d}_{_{b_ij}})\in\X_{\u_d}$ where $\u_d$ is a $\mu$-tableau, $\mu=(n-m+1,m-1)$,
obtained from $\t$ by moving the number $b_d$ to the
first row at the appropriate place to make the resulting tableau row-standard, then from (\ref{nbij}) we have:
\begin{equation}\label{nbijlemma}n^{d}_{_{b_ij}}=
\begin{cases}m_{_{b_ij}}&  \text{ if } i\leqslant d-1 \text{ or }j>b_d;\\
\alpha_i\in \F_q & \text{ if }  d+1\leqslant i \leqslant m, j = b_d;\\
m_{_{b_{i}j}}+\alpha_i m_{_{b_dj}}& \text{ if } d+1\leqslant i \leqslant m, j <b_d.\end{cases}\end{equation}
We split the summation in (\ref{phimr}) as follows:
\begin{equation}\label{phimr2}\Phi_m^d([M])=\sum\limits_{d+1\leqslant i \leqslant m \atop { b_i\in \ppi,\,\alpha_{i}\in \F_q}}\,\sum\limits_{d+1\leqslant u \leqslant m\atop{ b_u\notin \ppi,\,\alpha_{_u}\in \F_q}}[N^d_{\underline\alpha}(M)].\end{equation}
where $\underline \alpha=(\alpha_{_{d+1}},\ldots,\alpha_{_m})\in \F_q^{m-d}.$
Fixing $\alpha_i\in \F_q$ for all  $d+1\leqslant i \leqslant m$ satisfying $b_i\in \ppi$, let
\begin{equation}\label{nbar}\overline{N^d_{\underline{\check\alpha}}}(M)=
\sum\limits_{d+1\leqslant u \leqslant m\atop {\alpha_{u}\in \F_q,\, b_u\notin \ppi}}[N^d_{\underline\alpha}(M)]\end{equation}
 where $\check \alpha=(\alpha_{{i_1}},\ldots,\alpha_{i_h})$ with $d+1\leqslant i_1<\cdots< i_h \leqslant m, \,b_{{i_1}},\ldots,b_{i_h}\in \ppi$ for some $0\leqslant h \leqslant m-d;$ and then we can
 we can rewrite (\ref{phimr2}) as:
\begin{equation}\label{phimr3}\Phi_m^d([M])=\sum\nolimits_{\check \alpha \in \F_q^{h}}
\overline{N^d_{\underline{\check\alpha}}}(M).\end{equation}
Note that if for all $b_i\in \ppi$ all entries of $b_i$-th row  in $N^d_{\underline \alpha}(M)$ are zeros except the last 1's then
$\alpha_i=0$ for all $d+1\leqslant i \leqslant m$ such that $b_i\in \ppi$ \($\alpha_i$ is the entry at position $(b_i,b_d)$ of $N^d_{\underline \alpha}(M)$\) and hence by (\ref{nbijlemma}), we obtain $m_{_{b_ij}}=0$ for $(b_i,j)\in \J_{\t}$ and
$b_i\in \ppi.$ In this case $N^d_{\underline \alpha}(M)$ is a summand of $\overline{N^d_{\underline{\check 0}}}(M).$
Inserting (\ref{phimr3}) into (\ref{phielprop}), we obtain
\begin{eqnarray}\nonumber\Phi_m^d(e)&=&\frac{1}{q^{|\mathfrak{J}_{\t}|}}
\sum\limits_{M\in\X_{\t}}\prod\limits_{(b_i,j)\in\J_{\t}}
\theta(-l_{_{b_ij}}m_{_{b_ij}})\sum\limits_{\check \alpha \in \F_q^{h}}\overline{N^d_{\underline{\check\alpha}}}(M)\\
&=&\frac{1}{q^{|\mathfrak{J}_{\t}|}}
\sum\limits_{M\in\X_{\t}^0}\prod\limits_{(b_i,j)\in\J_{\t}\atop b_i\notin \ppi}
\theta(-l_{_{b_ij}}m_{_{b_ij}})\overline{N^d_{\underline{\check 0}}}(M)+y^d\label{n0yr}\quad\quad\end{eqnarray}
where
$\X_{\t}^0$ is the set of matrices $M\in \X_{\t}$ such that $m_{_{b_ij}}=0$ for all $(b_i,j)\in \J_{\t}$ with $ b_i\in \ppi$
 and $y^d(=y_r^d)$ is a linear combination of matrices in $\X_{\u_d}$ with at least one nonzero entry at a position  $(b_i,j)\in \J_{\u_d}$ with $ b_i\in \ppi;$ moreover, we used $\theta(-l_{_{b_ij}}m_{_{b_ij}})=1$ for $m_{_{b_ij}}=0$ with $(b_i,j)\in \J_{\t}$ and $ b_i\in \ppi$.
Since $\tilde e=\tilde e_{_{\tilde L}}$ with $\tilde L=\R_\p(L)$, by \ref{jsplit} and \ref{perservecor}, we have \begin{equation} \label{phim=d}
            \Phi_m(e)=\sum_{1\leqslant d\leqslant m\atop {b_d\notin \ppi}}\Phi_m^d(e)\quad\text{ hence }\quad\Phi_m(\tilde e)=\sum_{1\leqslant d\leqslant m\atop {b_d\notin \ppi}}\Phi_m^d(\tilde e).
           \end{equation}
Note that  this two summations have the same index set but we should keep in mind that for different $r\in\{1,\dots,k\}$, the index set $\{1\leqslant d \leqslant m\,|\,b_d\notin \ppi\}$ can be different.
Similarly as (\ref{phielprop}) and  (\ref{phimr}) for $b_d\notin \ppi,$ $(1\leqslant d \leqslant m)$ we have:
 \begin{eqnarray} \Phi_{m-s}^d(\tilde e)=\frac{1}{q^{|\mathfrak{J}_{\tilde \t}|}}
 \sum\limits_{\tilde M\in\X_{\tilde \t}}\prod\limits_{(b_i,j)\in\J_{\tilde \t}}
\theta(-l_{_{b_ij}}\tilde m_{_{b_ij}})\Phi_{m-s}^d([\tilde M])\label{phitilde}
\end{eqnarray}
and
 \begin{equation}\Phi_{m-s}^d([\tilde M])=\sum\limits_{\underline \beta\in \F_q^{m-d-h}}[\tilde N^d_{\underline\beta}(\tilde M)]\end{equation}
 where $\underline \beta=(\beta_{_{i_1}},\ldots,\beta_{i_{m-d-h}})\in \F_q^{m-d-h}$ with $d+1\leqslant i_t \leqslant m$
 such that $b_{i_t}\notin \ppi$.
Recall from (\ref{nbar}) we have \begin{equation}\nonumber\overline{N^d_{\underline{\check 0}}}(M)=\sum\limits_{ d+1\leqslant u \leqslant m\atop{\alpha_{_u}\in \F_q,\, b_u\notin \ppi}}[N^d_{\underline\alpha}(M)]\end{equation}
where $\underline \alpha=(\alpha_{_{d+1}},\ldots,\alpha_{_m})$ such that
$\alpha_i=0$ for all $d+1\leqslant i \leqslant m$ with $b_i\in \ppi.$
For all $(b_i,j)\in \J_{\tilde \t},$ identifying $\tilde m_{_{b_ij}}$ with $m_{_{b_ij}}$ we obtain:
\begin{equation}\label{rsn0}\R_{\p}\(\overline{N^d_{\underline{\check 0}}}(M)\)=\sum\nolimits_{\underline \beta\in \F_q^{m-d-h}}[\tilde N^d_{\underline\beta}(\tilde M)]=\Phi_{m-s}^d([\tilde M]).\end{equation}
Inserting back the index $r$ in (\ref{n0yr}) and (\ref{phim=d}), we obtain
\begin{eqnarray*}
0&=&\sum_{1\leqslant r \leqslant k} \gamma_r  \Phi_m(e_{{r}})=\sum_{1\leqslant r \leqslant k} \gamma_r \big(\sum_{ {b_d\in \underline{\t_r}\setminus\ppi}}\Phi_m^d(e_r)\big)\\
 &=&\sum_{1\leqslant r \leqslant k} \gamma_r\Big( \sum_{{b_{d}\in \underline{\t_r}\setminus\ppi}} \frac{1}{q^{|\mathfrak{J}_{r}|}}
\sum\limits_{M_{r}\in\X_{r}^0}\prod\limits_{(b_i,j)\in\J_{r}\atop b_i\notin \ppi }
\!\!\!\theta\,(-l^r_{_{b_ij}}m^r_{_{b_ij}})\,\overline{N^d_{\underline{\check 0}}}(M_r)+y_r^d\Big)
\end{eqnarray*}
where $\J_r=\J_{\t_r}$ and $\X_{r}^0=\X_{\t_r}^0$.
Note that all matrices involved in $y^d_r$ are linearly independent of those involved in $\overline{N^{d'}_{\underline{\check 0}}}(M_{r'})$
for every  $1\leqslant d'\leqslant m,1\leqslant r' \leqslant k$ with $b_{d'}\in \t_{r'}\setminus\ppi$  since they differ in some row $b_i\in \ppi.$ Hence we have
\begin{eqnarray}\label{n0lin}
\sum_{1\leqslant r\leqslant k}\Big(\sum_{ {b_d \in \underline{\t_r}\setminus\ppi}}\frac{\gamma_r }{q^{|\mathfrak{J}_{\t_r}|}}
\sum\limits_{M_{r}\in\X_{r}^0}\prod\limits_{(b_i,j)\in\J_{r}\atop b_i\notin \ppi}\!\!\theta\,(-l^r_{_{b_ij}}m^r_{_{b_ij}})\,\overline{N^d_{\underline{\check 0}}}(M_r)\Big)=0.\end{eqnarray}
Acting by the $F$-linear map $\R_{\p}$ on both sides of  (\ref{n0lin}), from   (\ref{phim=d}) and (\ref{rsn0}) we obtain:
\begin{eqnarray}\label{rsphi}
\sum_{1\leqslant r \leqslant k}\frac{\gamma_r }{q^{|\mathfrak{J}_{\t_r}|}}
\sum\limits_{M_{r}\in\X_{r}^0}\prod\limits_{(b_i,j)\in\J_{r}\atop b_i\notin \ppi}\theta(-l^r_{_{b_i j}}m^r_{_{b_i j}})\Phi_{m-s}([\tilde M_r])=0\quad\quad\end{eqnarray}
where $\tilde m^r_{_{b_i j}}=m^r_{_{b_i j}},$ for all $(b_i,j)\in \J_{\tilde \t_r}.$
We split the product in (\ref{rsphi}) along the column indices as the following:
\begin{eqnarray*}\label{phimrlin}\sum_{1\leqslant r \leqslant k}\frac{\gamma_r }{q^{|\mathfrak{J}_{\t_r}|}}
 \sum\limits_{M_{r}\in\X_{r}^0} \Big(\prod\limits_{(b_u,v)\in\J_{r}\atop b_u\notin \ppi,\,v\in \pj}\!\!\!\!\!\! \theta\,(-l^r_{_{b_uv}}m^r_{_{b_uv}}) \,\, \cdot\!\!\!\!\prod
\limits_{(b_i,j)\in\J_{r}\atop b_i\notin \ppi,\,j\notin \pj}\!\!\!\!\! \theta\,(-l^r_{_{b_ij}}m^r_{_{b_ij}})\Big)\Phi_{m-s}([\tilde M_r])=0.\end{eqnarray*}
Since $\Phi_{m-s}([\tilde M_r])$ is independent of $m^r_{_{b_uv}}$ for all $(b_u,v)\in\J_{r}$ with $b_u\notin \ppi, v\in \pj$
then by identifying   $\tilde m^r_{_{b_ij}}$ with $m^r_{_{b_ij}}$ for all $(b_i,j)\in \J_{\tilde r}=\J_{\tilde \t_r}=\{(b_i,j)\in\J_{r}\,|\,b_i\notin \ppi,\,j\notin \pj\},$
we can separate the summation in the formula above as  follows:
\begin{eqnarray}
\sum_{1\leqslant r \leqslant k}\frac{\gamma_r }{q^{|\mathfrak{J}_{\t_r}|}}
\!\!\prod\limits_{(b_u,v)\in\J_{r}\atop b_u\notin \ppi,\,v\in \pj}\sum_{\mbox{\small $m$}^r_{_{b_uv}}
\in \F_q}\!\!\theta(-l^r_{_{b_uv}}m^{r}_{_{b_uv}})
\sum_{\mbox{\small $\tilde m$}^r_{_{b_ij}}\in \F_q\atop {(b_i,j)\in\J_{\tilde r}}}\prod\limits_{(b_i,j)\in\J_{\tilde  r}}\!\!\theta(-l^r_{_{b_ij}}
\tilde m^r_{_{b_ij}})\Phi_{m-s}([\tilde M_r])=0.\quad\label{mrlin}\end{eqnarray}
Using (\ref{phim=d}) and (\ref{phitilde}), we rewrite (\ref{mrlin}):
\begin{eqnarray}\label{tildelin}\sum_{1\leqslant r \leqslant k}\Big( \frac{\gamma_r }{q^{|\mathfrak{J}_{\t_r}|}}
\!\!\prod\limits_{(b_u,v)\in\J_{r}\atop b_u\notin \ppi,\,v\in \pj}\!\!\sum_{\mbox{\small $m$}^r_{_{b_uv}}
\in \F_q}\!\!\!\theta(-l^r_{_{b_u v}}m^r_{_{b_u v}})\,{q^{|\mathfrak{J}_{\tilde  r}|}}
\Big)\Phi_{m-s}(\tilde e_{{r}})=0.\quad\end{eqnarray}
For $r \in\{1,\dots, k\}$, let
\begin{equation}\label{lastlin}\delta_r=\gamma_r \,\frac{q^{|\mathfrak{J}_{\tilde r}|}}{q^{|\mathfrak{J}_{r}|}}
\prod\limits_{(b_u,v)\in\J_{r}\atop b_u\notin \ppi,\,v\in \pj}\sum\limits_{\mbox{\small $m$}^{r}_{_{b_uv}}
\in \F_q}\theta(-l^r_{_{b_uv}}m^r_{_{b_uv}}),
\text{ then }
\sum\limits_{1\leqslant r \leqslant k} \delta_r \Phi_{m-s}(\tilde e_{{r}})=0.
\end{equation}
We can now assume $e_{_1}$ is a pattern idempotent with $\gamma_1 \not=0,$ since $\Phi_m$ is $FU$-linear and $U^w\cap U$ acts monomially on $\E_{\t_1}$ with $\t_1=\t^\lambda w$.
Hence we have $l^1_{_{b_u v}}=0$ for all $(b_u ,v )\in\J_{\t_1}$ with
$b_u \notin \ppi,\,v \in \pj.$
Therefore
$\delta_1=\gamma_1  q^{|\mathfrak{J}_{\tilde \t_1}|-|\mathfrak{J}_{\t_1}|+c}\not=0$
where $c=\big|\{(b_u,v)\in\J_{\t_1}\,|\,b_u\notin \ppi,\,v\in \pj\}\big|.$
Hence by (\ref{lastlin}), the set $\{\Phi_{m-s}(\tilde e_{{r}})\,|\,1\leqslant r \leqslant k\}$ is linearly dependent.
\end{proof}
We state two easy consequences obtained from the proof of \ref{compatible}:
\begin{Cor}\label{patternlin}If $L$ is a pattern matrix with $\pf(L)=\pf$ and we have $\Phi_m(e_{_{L}})+\sum_{K} \gamma_{_K}\, \Phi_m(e_{_{K}})=0$
with $\pf(\O_K)=\pf,$ then there exist $\delta_K\in F$ such that
$\Phi_{m-s}(e_{_{\R_{\p}(L)}})+\sum_{K} \delta_{_K}\, \Phi_{m-s}(e_{_{\R_{\p}(K)}})=0.$
\end{Cor}
%
\vspace{1mm}
Recall the definition of $T_{\p}^\lambda$ in \ref{tplambda}, then we have:\vspace{-1mm}
\begin{Cor}\label{compatible1}Let $\pf$ be a filled pattern associated with
some $\lambda$-pattern $\p$ with \mbox{$0\leqslant |\p| \leqslant m-1$.} 
If \mbox{ $\{\Phi_m(e_{\!_{L}})\,|\,e_{_{L}}\in \CC_{\pf}^\lambda,\, \tab(L)\setminus(\ppi\cup \pj)\in T_{\p}^\lambda\}$}
is    linearly dependent   
then 
$\{\Phi_{m-|\p|}(e_{_{\R_{\p}(L)}})\,|\,e_{_{L}}\in \CC_{\pf}^\lambda,\, \tab(L)\setminus(\ppi\cup \pj)\in T_{\p}^\lambda\}$
 is  linearly dependent.
\end{Cor}
%
%
%
\begin{Prop}\label{indep}
Let $\pf$ be a filled pattern associated with
some $\lambda$-pattern $\p$ with \mbox{$0\leqslant |\p| \leqslant m-1.$}
If $\Char(F)=0$, then
\[
\mathfrak{C}_{\pf}^\mu=F \Span \{\Phi_m (e_{_L})\,|\,\pf(\mathcal{O}_L)=\pf,
\tab(L)\setminus(\ppi\cup \pj)\in T^\lambda_{\p}\}
\]
as $F$-vector space, where $\mu=(n-m+1,m-1)\vdash n$.
\end{Prop}
\begin{proof}
Let $M_{\pf}=\{\Phi_m (e_{_L})\,|\,\pf(\mathcal{O}_L)=\pf,
\tab(L)\setminus(\ppi\cup \pj)\in T^\lambda_{\p}\}$. Obviously by  \ref{keepcon1} we have
$M_{\pf}\subset \mathfrak{C}_{\pf}^\mu.$
We prove first that $M_{\pf}$ is a linearly independent set.

Suppose $M_{\pf}$ is a linearly dependent set then
by Corollary \ref{compatible1}, $\{\Phi_{m-|\p|} (e_{_{\R_{\p}(L)}})\,|\,\pf(\mathcal{O}_L)=\pf,
\tab(L)\setminus(\ppi\cup \pj)\in T^\lambda_{\p}\}$ is  a linearly dependent set.
Assume $\sum \delta_{_{L}} \Phi_{m-|\p|} ( e_{_{\R_{\p}(L)}})=0$  with
$\pf(\mathcal{O}_L)=\pf,$ \mbox{$\tab(L)\setminus(\ppi\cup \pj)\in T^\lambda_{\p}$} and there
exist at least one $ L$ such that $\delta_{_{L}}\not=0.$ Then
\mbox{$0\not=\sum \delta_{_{L}} e_{_{\R_{\p}(L)}}\in \ker \Phi_{m-|\p|}.$}
For $\Char(F)=0$, we have by \ref{sphim}, $S^{(n-m-|\p|,m-|\p|)}=\ker \Phi_{m-|\p|}.$ Thus
$0\not=\sum \delta_{_{L}} e_{_{\R_{\p}(L)}}\in  S^{(n-m-|\p|,m-|\p|)}$.    
Moreover by \ref{shifted}, we know $\tab\big(\R_{\p}(L)\big)$ is row-standard but non-standard
since $\tab(L)\setminus(\ppi\cup \pj)\in T^\lambda_{\p}.$ This is a
contradiction to  \ref{lyle}.
Thus $M_{\pf}$ is a linearly independent set and then we shall prove $|M_{\pf}|=\dim_F(\CC_{\pf}^\mu)$.

Let $M_\O$ (resp. $M_{\tilde \O}$) denotes orbit modules in $M^\lambda$ (resp. $M^\mu$). By \ref{iso}
for any $\O$ (resp. $\tilde \O$) such that $\pf(\O)=\pf$ (resp. $\pf(\tilde \O)=\pf$), we have
$\dim M_\O=\dim  M_{\tilde \O}:=a.$
Then $|M_{\pf}|=a\cdot|T^\lambda_{\p}|$ and $\dim_F(\CC_{\pf}^\mu)=a\cdot|\RStd(\mu)|.$
By \ref{tssize}, we know $|T^\lambda_{\p}|=|\RStd(\mu)|$ hence
$|M_{\pf}|=\dim_F(\CC_{\pf}^\mu).$ 
\end{proof}


In general, there exist some $L$ such that $\p(\O_L)=\p$ and  $\tab(L)$
is  standard but $\tab(L)\setminus \ppi\cup \pj$ is nonstandard.
\begin{Lemma}\label{non-standard}
Let $\p$ be a  $\lambda$-pattern.
If $\p$ fits some \mbox{$\t\in \RStd(\lambda)$} and $\t$ is non-standard, then $\tilde\t=\t\setminus(\ppi\cup \pj)$ is non-standard.
\end{Lemma}
\begin{proof}If $s=|\p|=0,$ the lemma holds obviously. Now we assume $s>0.$
Let $$\t=\begin{tabular}{|c|c|c|c|c|c|}\hline $a_1$ & $a_2$&$\cdots$&$a_m$&$\cdots$&$a_{n-m}$\\
\hline $b_1$ & $b_2$&$\cdots$&$b_m$\\\cline{1-4} \end{tabular}\, \in \RStd(\lambda)\setminus \Std(\lambda).$$  If one can prove for any $(b_i,j)\in \p$  that $\t\setminus \{b_i,j\}$ is non-standard, then the lemma  holds inductively. We leave the details to the readers.
\end{proof}

\begin{Lemma}\label{lr}
Let  $ \p$ be a  $\lambda$-pattern and  let $e_{_L}, e_{_R}\in M^\lambda $ such that
$\p(\O_L)=\p(\O_R)=\p,$ then $\tab(R)<\tab(L)$ implies $\tab(\tilde R)< \tab(\tilde L).$
\end{Lemma}
\begin{proof}Let $\underline {\t_1}=\underline{\tab(L)}=(t_1,t_2,...,t_m),\, \underline {\t_2}=\underline{\tab(R)}=(r_1,r_2,...,r_m)$ \vspace{1mm}
and assume $\t_2<\t_1.$ Working step by step, by removing one element in the pattern at each step we may assume that $\p=\{(k,j)\}$
consists of one element.
Since $\p$ fits $\t_1$ and $\t_2$,
$k\in \{t_1,t_2,\dots,t_m\}\cap\{r_1,r_2,\dots,r_m\}.$ Note that
\mbox{$\underline {\tilde \t_1}=(t_1,t_2,\dots,t_m)\setminus\{k\}$, $\underline {\tilde \t_2}=(r_1,r_2,\dots,r_m)\setminus\{k\}.$}
Assume $i$ is the smallest number satisfying $r_i < t_i$.
Then by the minimality of $i$, we obtain:
$k < r_i \text{ or } k \geqslant t_i .$
In fact, for $k < r_i \text{ or }k > t_i,$ it is easy to get $\tilde \t_2<\tilde \t_1.$ Here we only deal with the case $k=t_i.$
In this case, $ t_{i+1}> t_i > r_i$. And  we get
$\underline {\tilde \t_1}=(r_1,\dots,r_{i-1},t_{i+1},\dots,t_m),$ $\underline {\tilde \t_2}=(r_1,\dots,r_{i-1},r_{i},\dots,r_m)\setminus \{k\}$
 where $k=r_j$ such that $j>i.$ Hence we obtain $\tilde \t_2<\tilde \t_1.$
\end{proof}

\begin{Theorem}\label{main}Let $\Char(F)=0,\,\lambda=(n-m,m)\vdash n$. For \mbox{$e_{_L}\in \O\subset M^\lambda$} with $ \p=\p(\O)$
there exists  $v\in S^\lambda$ such that $\last(v)=\tab(L)$ and \mbox{$\ttop(v)=e_{_L}$} if and only if
$\tab(L)\setminus(\ppi\cup \pj) \text{ is a shifted standard $\mu$-tableau,}$ where $\mu=(n-m-|\p|,m-|\p|)$;
here ``shifted''   means the tableau is filled by numbers in $\{1,2,\ldots,n\}\setminus (\ppi\cup \pj).$
\end{Theorem}
\proof
Let $e_{_L}\in \O\subset M^\lambda$ with $\pf=\pf(\O),\,\p=\p(\O)$ and $s=|\p|.$ In particular, we have discussed two
special types of orbits in Section \ref{specialorbit}. One is the case of orbits with full pattern:
For  $s=|\p|=m,$ by Proposition \ref{fullcon} and Proposition \ref{sphim}, we have
$e_{_L}\in S^\lambda \text{ and }\tab(L)\in\Std(\lambda).$
In particular, in this case 
$\tab(L)\setminus(\ppi\cup \pj) \text{ is a shifted standard $(n-2m,0)$-tableau}.$ The other type of special orbits are those with $s=0$. In this case the sufficiency is  \ref{trivialorbit} and the necessity is  \ref{lyle}. Now we assume \mbox{$1\leqslant  s \leqslant m-1.$}
\begin{itemize}
\item[(1)]$(\Longleftarrow)$
Assume 
\mbox{$\tab(L)\setminus(\ppi\cup \pj)$} is a shifted standard $\mu$-tableau. By \ref{non-standard},
we know $\tab(L)$ is standard.
By \ref{indep}, \begin{equation}\label{thirdlin}
                 \Phi_m (e_{\!_L})= \sum {a_{_{R}}\Phi_m (e_{_R})}\in \mathfrak{C}_S^{\mu}
                \end{equation}
 where  $e_{_R}\in \O_R\subset M^\lambda$ with $\pf(\O_R)=\pf,\,$ \mbox{$\tab(R)\setminus(\ppi\cup \pj)\in T^\lambda_{\p}$} and $a_{_{R}}\in F$.
We claim that all occurring $R$ with $a_{_{R}}\not=0$ has the property: \mbox{$\tab(R)<\tab(L).$}
Otherwise, assume there exist  some $R$ such that $a_{_{R}}\not=0 \text{ and } \tab(R)>\tab(L).$
We choose some $u\in U^w\cap U$ where $\t^\lambda w=\tab(R)$ such that
$e_{_{R_0}}=e_{_{R}}\circ u$, is a pattern idempotent.
Hence we obtain:
$\Phi_m (e_{\!_L}\circ u)- a_{_{R}}\Phi_m (e_{_{R_0}})-\Phi_m \big(\sum_{_{R'\not=R}} a_{_{R'}}e_{_{R'}}\circ u\big)=0.$
Suppose $e_{_L}\circ u=\sum_{_K} \alpha_{_K} e_{\!_K},$ and
$\sum_{_{R'\not=R}} a_{_{R'}} e_{_{R'}}\circ u=\sum_{_N} \beta_{_N} e_{\!_N}.$
Then:
\begin{equation}\label{firstlin}
 \sum\nolimits_K \alpha_{_K}\Phi_m (e_{\!_K})- a_{_{R}}\Phi_m (e_{_{R_0}})-\sum\nolimits_N \beta_{_N} \Phi_m(e_{\!_N})=0
\end{equation}
where $\tab(K)=\tab(L), \,\tab(N)\setminus(\ppi\cup \pj)\in T_{\p}^\lambda$ for all $K$ and $N.$

Denote $\tilde K=\R_{\p}(K),\,\tilde R_0=\R_{\p}(R_0),\tilde N=\R_{\p}(N)$. Then  by \ref{patternlin} and (\ref{firstlin}), we obtain:
\begin{equation}\label{secondlin}
\sum\nolimits_K \delta_{_K}\Phi_{m-|\p|} (e_{\!_{\tilde K}})- a_{_{R}}\Phi_{m-|\p|} (e_{_{ {\tilde R_0}}})-\sum\nolimits_N \delta_{_N} \Phi_{m-|\p|}(e_{\!_{\tilde N}})=0
\end{equation}
where $\tab(\tilde K)=\tab(\tilde L), \tab(\tilde R_0)=\tab(\tilde R)$ and $\tab(\tilde N)$ is nonstandard.
Since $\Char(F)=0,$ by  \ref{sphim} we have $S^{\mu}=\ker \Phi_{m-|\p|}$ where $\mu=(n-m-|\p|, m-|\p|)$ and then from (\ref{secondlin}) we get:
$$0\not=x=\sum\nolimits_K \delta_{_K}e_{\!_{\tilde K}}- a_{_{R}}e_{_{ {\tilde R_0}}}-\sum\nolimits_N \delta_{_N} e_{\!_{\tilde N}}\in S^\mu .$$
By assumption we have $\tab(R)>\tab(L)$ then from Lemma \ref{lr}, we obtain
$\tab(\tilde R_0)=\tab(\tilde R)>\tab(\tilde L)=\tab(\tilde K).$
Moreover we know $\tab(\tilde R_0)$ and $\tab(\tilde N)$ are non-standard. Hence we obtain that $\last(x)$ is non-standard, which is a contradiction to
 \ref{lyle}.
Let $v=e_{_L}-\sum {a_{_{R}}e_{_R}}.$  By \ref{sphim} and (\ref{thirdlin}) we get
$v\in \ker\Phi_m=S^\lambda$ with $\last(v)=\tab(L),$ $\ttop(v)=e_{\!_L}.$
This finishes the proof of the sufficiency.
\item[(2)] $(\Longrightarrow)$
Suppose 
there exists  $v\in S^\lambda$ such that $\last(v)=\tab(L)$ and $\ttop(v)=e_{_L}.$
Assume $0\not=v=e_{\!_L}-\sum\nolimits_R {a_{_{R}}e_{_R}}\in S^\lambda=\ker \Phi_m$
where $\tab(R)<\tab(L)$  and $0\not=a_{_{R}}\in F.$ Thus \begin{equation}\label{v}\Phi_m(e_{\!_L})=\sum {a_{_{R}}\Phi_m(e_{_R}}).\end{equation}
By   \ref{keepcon1},  $\Phi_m(e_{\!_L})\in \CC_{\pf}^\mu,$
hence we can assume for all $R$ in (\ref{v}), we have  $\pf(\O_R)= \pf$ where $\O_R$ denotes the orbit containing $R$;
moreover, from the proof in (1), we can further assume:  
$\tab(R)\setminus(\ppi\cup \pj)\in T^\lambda_{\p}.$

In fact it suffices to prove for $L$ is a pattern matrix since $\tab(K)=\tab(L)$ for all $e_{_K}\in \O_L.$
Assume $L$ is a pattern matrix, then by  \ref{patternlin} and (\ref{v}),
$\Phi_{m-|\p|}( e_{_{\tilde L}})= \sum\nolimits_{_R} {b_{_{R}} \Phi_{m-|\p|}( e_{_{\tilde R}})}$
for some $b_{_{R}} \in F.$
That is,
\mbox{$v_{_{\tilde L}}:=e_{_{\tilde L}}-\sum b{_{_R}}   e_{_{\tilde R}}\in \ker\Phi_{m-|\p|} = S^{\mu}$}
where \mbox{$\mu=(n-m-|\p|,m-|\p|).$}
By \ref{lr}, we have \mbox{$\tab(\tilde R)< \tab(\tilde L)$} and from   \ref{lyle}, we obtain
$\tab(\tilde L)=\last(v_{_{\tilde L}})\in \Std(\mu).$ Thus by   \ref{shifted},
$\tab(L)\setminus(\ppi\cup \pj)$  is  a shifted standard $\mu$-tableau. This finishes the proof of the necessity.\qed\end{itemize}

\begin{Defn}\label{basisfor0}
Suppose $\Char F=0.$ Let $\lambda=(n-m,m)\vdash n$ and $e_{_L}\in \O \subset M^\lambda$ with
$\pf(\O)=\pf$ and $\p(\O)=\p$. Suppose \mbox{$\tab(L)\setminus(\ppi\cup \pj)$} is  a shifted standard $\mu$-tableau with \mbox{$\mu=(n-m-|\p|, m-|\p|).$} Choose $v\!_{_L}\in S^\lambda$ such that $\last(v\!_{_L})=\tab(L)$ and $\ttop(v\!_{_L})=e_{_L}.$ (By  \ref{main} there exists such an $v\!_{_L}$). Let
$$\mathcal B ^\lambda_{\pf}:=\mathcal B ^\lambda_{\pf,_F}=\{v\!_{_L}\,|\,e_{_L}\in \O\subset M^\lambda, \pf(\O)=\pf, \tab(L)\setminus(\ppi\cup \pj) \text{ is  standard}\}.$$ Finally   take
$\mathcal B ^\lambda=\mathcal B ^\lambda_{_F}=\dot{\bigcup\limits_{\pf}}\,\mathcal B ^\lambda_{\pf}\,.$
Note that this union is disjoint, since its elements are distinguished by their leading term $\ttop(v\!_{_L})=e_{_L}$ and we say this $e_{_L}$  appears as leading term of $S^\lambda$.
\end{Defn}

We choose now a suitable principal ideal domain $\Lambda$ (containing a primitive $p$-th root of unity), with quotient field $Q$
of characteristic zero. Moreover we assume that $q=q\cdot 1_\Lambda\in \Lambda$ is invertible. Finally
We assume that  our field $F$ is epimorphic image of $\Lambda$ and has characteristic $l$ coprime to $q.$
Note that
$M^\lambda_R=M^\lambda_\Lambda\otimes_\Lambda R\,\text{ and } S^\lambda_R=S^\lambda_\Lambda\otimes_\Lambda R\text{ for }R=Q \text{ or } F.$

\begin{Prop}\label{intbasis}In the notation of  \ref{basisfor0}, replacing $F$ by $Q$, we have  $0\not=v_{_L}\in S^\lambda_\Lambda$ and
$v_{_{L,F}}=v_{_L}\otimes_{_\Lambda} 1_{_F}\not=0$ with $\ttop(v_{_{L,F}})=e_{_L}.$
\end{Prop}
\begin{proof}Note that $e_{_{L,Q}}=e_{_{L,\Lambda}}$. Keeping notation in \ref{basisfor0}, by \ref{indep} we may write uniquely
$\Phi_m(e_{_L})+\sum\nolimits_K \alpha_{_K}\Phi_m(e_{_K})=0$
where $K$ runs through all matrices with $\pf(\O_K)=\pf(\O_L):=\pf$ and \mbox{$\tab(K)\setminus (\ppi\cup \pj)\in T_{\p}^\lambda$} and $\alpha_{_K}\in Q.$ Thus
$e_{_L}+\sum\nolimits_K \alpha_{_K}e_{_K}\in S^\lambda_Q=\ker \Phi_{m,Q}.$
Multiplying  this equation by the least common denominator of the coefficients we obtain an expression
\begin{equation}\label{hatvl}\hat v_{_L}:=\beta_{_L}e_{_L}+\sum\nolimits_K \beta_{_K}e_{_K}\in \ker \Phi_{m,\Lambda}\text{ with }\beta_{_L}, \beta_{_K}\in \Lambda,\,\forall\,K.\end{equation}
Moreover we may assume that the greatest common divisor of the coefficients $\beta_{_L},\, \beta_{_K}$ is $1$.
Note that $\hat v_{_L}\in S^\lambda_\Lambda$ hence
$\hat v_{_{L,F}}=\hat v_{_L}\otimes_\Lambda 1_F \in S_F^\lambda.$
Let \begin{equation}\label{hatv}\hat v_{_{L,F}}=\hat v_{_L}\otimes_\Lambda 1_F=\overline {\beta_{_L}}e_{_L}+\sum\nolimits_K \overline {\beta_{_K}}e_{_K},\end{equation}
where for $c\in \Lambda$, $\overline c$ denoted the corresponding residue class of $c$ in $F.$
Here we identify $M^\lambda_F=M^\lambda_\Lambda/l M^\lambda_\Lambda$
and $M^\lambda_F=M^\lambda_\Lambda\otimes_\Lambda F$ by the canonical isomorphism, where $l\in \Lambda$ generates the kernel of the epimorphism from $\Lambda$ onto $F.$
Since we have assumed the greatest common divisor of the coefficients $\beta_{_L},\, \beta_{_K}$ is $1$, we obtain
$\hat v_{_{L,F}}\not=0 \text{ and }\hat v_{_{L,F}}\in S_F^\lambda=S_\Lambda^\lambda\otimes_\Lambda F.$

We claim that $\overline {\beta_{_L}}\not=0$ in $F.$ Otherwise, if  $\overline {\beta_{_L}}=0$ in $F$ then in (\ref{hatv}) there exist some $K,$ namely $R$, such that
\begin{equation}\label{r0}\overline {\beta_{_R}}\not=0\text{ in }F\text{ and } \tab(R)\setminus (\ppi\cup \pj)\in T_{\p}^\lambda.\end{equation}
Acting by a suitable $u\in U$ we can obtain a pattern matrix $R_0$ such that
\begin{equation}\label{r0r}e_{_R}\circ u=e_{_{R_0}} \text{ with }\tab(R_0)=\tab(R).\end{equation}
Denote $e_{_L}\circ u=\sum\nolimits_{_X} a_{_X}e_{_X}$ and  $\sum_{_{K\not=R}} \beta_{_K} e_{_K}\circ u=\sum_Y b_{_Y}e_{_Y}$ where
\begin{equation}\label{y}0\not=a_{_X}, b_{_Y}\in \Lambda \text{ and } \tab(Y)\setminus (\ppi\cup \pj)\in T_{\p}^\lambda.\end{equation}
Thus by (\ref{hatvl}) we obtain
$0\not=\hat v_{_L}\circ u
=\beta_{_L} \sum_X a_{_X}e_{_X}+\beta_{_R}e_{_{R_0}}+\sum_Y b_{_Y}e_{_Y}\in S_\Lambda^\lambda=\ker \Phi_{m,\Lambda}$
 and then by \ref{patternlin} we get:
$\beta_{_L} \sum_X a_{_X} \delta_{_X}\Phi_{m-|\p|,\Lambda}(e_{_{\R_S(X)}})+
\beta_{_R}\Phi_{m-|\p|,\Lambda}(e_{_{\R_\p(R_0)}})+\sum_Y b_{_Y}\delta_{_Y}\Phi_{m-|\p|,\Lambda}(e_{_{\R_\p(Y)}})=0$
where $\delta_{_X},\delta_{_Y}$ are just zeros or some powers of $q$ by construction. Let
\begin{equation}\label{z}z=\beta_{_L} \sum\nolimits_{_X} a_{_X} \delta_{_X}e_{_{\R_\p(X)}}+
\beta_{_R}e_{_{\R_\p(R_0)}}+\sum\nolimits_{_Y} b_{_Y}\delta_{_Y}e_{_{\R_\p(Y)}}\end{equation} then
$0\not= z\in \ker\Phi_{m-|\p|,\Lambda} =S^\lambda_\Lambda.$
Since $\overline {\beta_{_R}}\not=0\text{ in }F$ we obtain $z\otimes_\Lambda 1_F\not=0$ in $F$. Moreover, we have
$0\not= z\otimes_\Lambda 1_F\in S^\lambda_F=S^\lambda_\Lambda\otimes_\Lambda F.$
Since $\overline {\beta_{_L}}=0$, from (\ref{z}) we obtain
$0\not=z\otimes_\Lambda 1_F=\overline{\beta_{_R}}e_{_{\R_\p(R_0)}}+\sum_Y \overline{b_{_Y}}\overline{\delta_{_Y}}e_{_{\R_\p(Y)}}\in S^\lambda_F.$
Hence by (\ref{r0}), (\ref{r0r}), (\ref{y}) and  \ref{shifted}, we know
$z\otimes_\Lambda 1_F$ is a nonzero element of $S^\lambda_F$ with $\last(z\otimes_\Lambda 1_F)$ non standard, which is a contradiction
to \ref{lyle}. So $\overline {\beta_{_L}}=\beta_{_L}\otimes_\Lambda 1_F\not=0$ in $F.$

This means that $\Char(F)=l$ does not divide $\beta_{_L}\in \Lambda.$ Choosing for example $\Lambda$ to be the integral
closure of $\Z$ in the field $\Q[\varepsilon]$, where $\varepsilon$ is a $p$-th root of unity, we may vary $l$ through all primes of $\Z$ except $p$ to conclude that $\beta_{_L}$ must be a unit in $\Lambda$. Thus we can assume $\beta_{_L}=1$. This shows:
\mbox{$v_{_{L,\Lambda}}=e_{_L}+\sum_{_K} \alpha_{_K}e_{_K}\in S^\lambda_\Lambda$} and
\mbox{$v_{_{L,F}}=e_{_L}+\sum_K \overline{\alpha_{_K}}e_{_K}\in S^\lambda_F$}  with  $\ttop(v_{_{L,F}})=e_{_L}.$
\end{proof}

We remark that if $e_{_L}\in \O$ can appear as a leading term of $S^\lambda$, then all the idempotents in $\O$ can also be a leading term of $S^\lambda$, thus we say $M_\O$ appears as a leading term. Now we can state the main result of this thesis.
\begin{Theorem}\label{main2} Let $\lambda=(n-m,m)\vdash n.$ Then $ \mathcal B ^\lambda$ is an integral  standard basis for $S^\lambda$ and
$\mathcal B ^\lambda_{\pf}$ is an integral  standard basis of the $\pf$-component $S^\lambda\!\! \downarrow_{\pf}$ of $\RRes^{\tiny{\mbox {$FG$}}}_{\tiny{\mbox {$FU$}}} S^\lambda$. Moreover for  $0\leqslant c\in \Z,$ there exist polynomials $f_c(t)\in \Z[t]$
such that the number of irreducible components of $\RRes^{\tiny{\mbox {$FG$}}}_{\tiny{\mbox {$FU$}}} S^\lambda$
of dimension $q^c$ is $f_c(q)$. Here  $S^\lambda$ is over any field  $F$  with characteristic coprime to $p$ containing a primitive $p$-th root of unity.\end{Theorem}
\begin{proof}Obviously, $\B^\lambda$ is linearly independent subset of $S^\lambda$. And by Theorem \ref{main} and Proposition \ref{intbasis} we have $|\B^\lambda|=\dim_{_F} S^\lambda.$ Hence the first statement holds.
Moreover, if $M_\O$ appears as the leading term of $S^\lambda$  then  varying the filling of the pattern $\p=\p(\O)$, we will get  $(q-1)^{|\p|}$ many orbit $M_{\O'}$ appearing as the leading term of $S^\lambda$ with $\p(\O')=\p$. And obviously the dimensions of these orbits are the same, given by the hook length. That is for $0\leqslant c\in \Z,$ there exist polynomials $f_c(t)\in \Z[t]$
such that the number of irreducible components of $\RRes^{\tiny{\mbox {$FG$}}}_{\tiny{\mbox {$FU$}}} S^\lambda$
of dimension $q^c$ is $f_c(q)$.
\end{proof}

\subsection{Rank polynomials $r_\t(q)$}
In  \cite{brandt2}, Brandt-Dipper-James-Lyle   introduced a kind of polynomials in $q$ attached to each standard $\lambda$-tableau $\t,\,\big(\lambda=(n-m,m)\vdash n\big)$, called ``rank polynomials'', denoted by $r_\t(q)$ such that $r_\t(1)=1.$
We will show that the number of our basis elements $B^\lambda$ in the $\t$-batch $\M_\t$
with leading term in $\E_\t$ is exactly the rank polynomial $r_\t(q)$.

\begin{Defn}(Brandt, Dipper, James and Lyle \cite{brandt2})
\begin{itemize}
\item[(1)] Consider a rectangular $a\times b, \,(a\leqslant b)$ array of boxes embedding into a $\Z\times \Z$
coordinate system such that the northwest corner has coordinate $(1,1)$. 
For example, $a=5,\, b=8:$
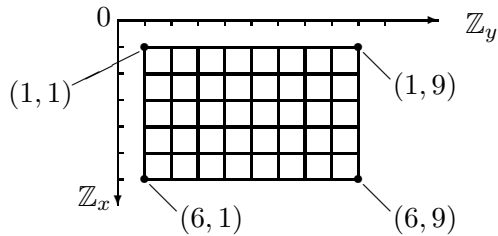
\begin{center}
\begin{picture} (150,80)
\put(0,80){\vector(0,-1){70}}
\put(0,80){\vector(1,0){120}}
\put(-15,10){$ \Z_x$}
\put(130,75){$ \Z_y$}
\put(-8,78){$0$}
\multiput(0,70)(0,-10){6}{\line(1,0){2}}
\multiput(0,80)(10,0){11}{\line(0,-1){2}}
\linethickness{0.8pt}
\multiput(10,69.8)(0,-10){6}{\line(1,0){80}}
\multiput(10,69.8)(10,0){9}{\line(0,-1){50}}
\put(10,20){\circle*{3}}
\put(90,20){\circle*{3}}
\put(10,70){\circle*{3}}
\put(90,70){\circle*{3}}
\put(12,18){\line(1,-1){10}}
\put(8,68){\line(-2,-1){25}}
\put(-41,49){$(1,1)$}
\put(23,2){$(6,1)$}
\put(92,18){\line(1,-1){10}}
\put(103,2){$(6,9)$}
\put(92,68){\line(1,-1){10}}
\put(103,52){$(1,9)$}
\end{picture}
\end{center}
We call a route along the grids from the northwest   corner to the southeast corner a \textbf{path}, denoted by $\pi$.
Define $P(a,b)$ to be the set of all paths in an $a\times b$ array of boxes.
\vspace{-2mm}
\item[(2)] Given a corner $(i,j)$ let $r(i,j)=j-i$. Suppose that $Y$ is a filling of the boxes to the south of some path $\pi$
with elements of $\F_q.$ Say that $Y$ is good if for each corner $(i,j)$ through which the path passes, the matrix whose bottom left hand corner is $(a+1,1)$ and whose top right hand corner is $(i,j)$ has rank at most $r(i,j)$.
\vspace{-2mm}
\item[(3)] We define the rank polynomial $r(\pi)$ of the path to be the number of ways of filling the boxes below the path with elements of $\F_q$ such that the filling is good.
\end{itemize}
\end{Defn}

\begin{Remark}\label{eastmoving}(Brandt, Dipper, James and Lyle \cite{brandt2})
\begin{itemize}
\item[(1)]
If $\pi$ passes through a corner with $i>j$ then $r(\pi)=0$. In particular, if $r(\pi)\not=0$ then the path must start with a east move.
\vspace{-3mm}
\item[(2)] Note that in the definition of a good filling, we may replace  `for each corner $(i,j)$ through which the path passes'
by `for each corner $(i,j)$ through which the path passes and which has the property that $(i-1,j)$ and $(i,j+1)$ are on the path'
since all the other restrictions follow from these.
\end{itemize}
\end{Remark}


\begin{Lemma}\label{std}Let $\lambda$ be a two part partition and \mbox{$\t \in \RStd(\lambda)$}.
Then \mbox{$\t \in \Std(\lambda)$} if and only if all the corners $(i,j)$ of ${\pi}_{\t}$ satisfying $i\leqslant j.$
\end{Lemma}
\begin{proof}Suppose $\t=\begin{tabular}{|c|c|c|c|c|c|}\hline $a_1$ & $a_2$&$\cdots$&$a_m$&$\cdots$&$a_{n-m}$\\
\hline $b_1$ & $b_2$&$\cdots$&$b_m$\\\cline{1-4} \end{tabular}\,\in  \RStd(\lambda).$
Then \begin{equation}\label{stdij}\t \in \Std(\lambda) \Leftrightarrow b_i>a_i,\,\forall\, i=1,\ldots,m.\end{equation}
If we label the boxes by their left top corner labeling, then (\ref{stdij})
 is equivalent to say box $(i,i)$ appears in the south of the path ${\pi}_{\t},$ that is
all the corners $(i,j)$ of ${\pi}_{\t}$ satisfying $i\leqslant j.$
\end{proof}

\begin{Theorem}\label{ranlpolynomial}Let $\lambda$ be a two part partition and $\t \in \RStd(\lambda)$.
Let \mbox{$e_{_L}\in \O\subset \M_\t\subset M^\lambda$} with $\p(O)=\p.$
Then $L \text{ is a good filling of path }{\pi}_{\t} $ if and only if $\t \setminus (\ppi\cup \pj)\text{ is standard}.$
\end{Theorem}
\begin{proof}By Remark \ref{eastmoving}, the path ${\pi}_{\t}$ must start with a east move, hence we can draw the following picture
for it:
\begin{center}
\begin{picture}(200,90)
\put(60,80){\line(1,0) {40}}
\put(100,80){\line(0,-1) {15}}
\multiput(60,65)(5,0){8}{\line(1,0){2}}
\multiput(100,65)(0,-5){8}{\line(0,1){2}}
\put(102,68){\tiny $(i,j)$}
\put(100,65){\circle*{5}}
\put(100,65){\line(1,0) {25}}
\put(125,65){\line(0,-1) {15}}
\put(125,50){\circle*{5}}
\put(125,50){\line(1,0) {25}}
\put(150,50){\line(0,-1) {20}}
\put(60,45){ $ M_{\mbox {\tiny $(i,j)$}}$}
\put(70,12){\textbf{Picture of } ${\pi}_{\t}$}
\end{picture}
\end{center}
\vspace{-5mm}
Note that $\t \setminus (\ppi\cup \pj)$ is a shifted tableau, filled by numbers in \mbox{$\{1,2,\ldots,n\}\setminus (\ppi\cup \pj).$}
Remember the definition of $\p$-similar in  \ref{shifted}. We denote
$\s\,\stackrel{\p}{\sim}\, \t \setminus (\ppi\cup \pj).$
Thus after deleting the rows and columns which contain positions in $\p,$ and closing the gaps, we obtain the path ${\pi}_\s.$
Again by  \ref{eastmoving}, it is sufficient to investigate those kind of corners labeled by  black dots in the Picture of ${\pi}_{\t}$ above. Choose an arbitrary corner of this kind, say $(i,j).$ Note that $L$ is obtained from a pattern matrix $L_0$ by truncated row and column operations. Furthermore note that such operations preserve the ranks of the sub-matrices determined by the relevant corners of the path ${\pi}_{\t}$. In particular $L$ is a good filling if and only if $L_0$ is a good filling. Thus we may assume that $L$ is a pattern matrix.

Assume there are $\alpha_{_{(i,j)}}$ many positions in the north west boxes $(u,v)$ of the corner $(i,j)$ such that $(u,v)\in \p,$ and $\beta_{(i,j)}$ many positions in the south west boxes $(s,t)$ of the corner $(i,j)$ such that $(s,t)\in \p.$
If we denote the south west part of the corner $(i,j)$ by $M_{(i,j)},$ then
by  \ref{calcor}, we obtain that:
\begin{equation}\label{rank}\rank M_{(i,j)}=\beta_{(i,j)}.\end{equation}
Hence after deleting the rows and columns which contain positions in $\p,$ and closing the gaps, the corner $(i,j)$
has a new labeling $(i',j')$ namely
$i'=i-\alpha_{_{(i,j)}}\text{ and } j'=j-\alpha_{_{(i,j)}}-\beta_{(i,j)}.$ Hence
\begin{equation}\label{j'}j'-i'=j-\alpha_{_{(i,j)}}-\beta_{(i,j)}-\(i-\alpha_{_{(i,j)}}\)=j-i-\beta_{(i,j)}.\end{equation}
By the definition of good filling,
$L\text{ is a good filling of }{\pi}_{\t}$ if and only if $\rank M_{(i,j)}\leqslant j-i$ for all black dots $(i,j)$.
By (\ref{rank}), we get
\begin{equation}\label{rbe}\rank M_{(i,j)}\leqslant j-i \Leftrightarrow \beta_{(i,j)} \leqslant j-i.\end{equation}
Combining (\ref{j'}) and  (\ref{rbe}), we get
\begin{equation}\rank M_{(i,j)}\leqslant j-i \Leftrightarrow j'-i'\geqslant 0\Leftrightarrow i'\leqslant j'.\end{equation}
By \ref{std}, we get that $\s$ is a standard tableau. Then by  \ref{shifted}, we obtain that $\t \setminus (\ppi\cup \pj)$
is standard.
\end{proof}

Combining the two main results \ref{main2} and \ref{ranlpolynomial}, we actually get a reproof of the following theorem:
\begin{Theorem}(Brandt, Dipper, James and Lyle \cite{brandt2})\label{brandt}\\
If $L$ is a good filling for ${\pi}_{\t}$ where $\t=\tab(L)$, then there   exist $v_{_L}\in S^\lambda$ such that $\ttop(v_{_L})=e_{_L}$ and $\last(v_{_L})=\tab(L).$ Moreover, if we choose some appropriate $v_{_L}$ for each $L$ which is a good filling, then
$\{v_{_L}\,|\,L \text{ is a good filling}\} $ is a standard basis of $S^\lambda$.
\end{Theorem}


\providecommand{\bysame}{\leavevmode ---\ }
\providecommand{\og}{``} \providecommand{\fg}{''}
\providecommand{\smfandname}{and}
\providecommand{\smfedsname}{\'eds.}
\providecommand{\smfedname}{\'ed.}
\providecommand{\smfmastersthesisname}{M\'emoire}
\providecommand{\smfphdthesisname}{Th\`ese}

\end{document}